\documentclass[]{siamart}

\usepackage{lipsum}
\usepackage{amsfonts}
\usepackage{graphicx,epstopdf}
\usepackage[caption=false]{subfig}
\usepackage{epstopdf}
\usepackage{algorithmic}
\ifpdf
  \DeclareGraphicsExtensions{.eps,.pdf,.png,.jpg}
\else
  \DeclareGraphicsExtensions{.eps}
\fi

\newcommand{\TheTitle}{On the arbitrarily long-term stability of conservative methods} 
\newcommand{\TheAuthors}{Andy T. S. Wan, and Jean-Christophe Nave}

\headers{\TheTitle}{\TheAuthors}

\title{{\TheTitle}\thanks{This work was supported by the NSERC Discovery program and the Centre de Recherches Math\'{e}matiques.}}

\author{
  Andy T. S. Wan\thanks{Department of Mathematics and Statistics, McGill University, Montr\'eal, QC, H3A 0B9, Canada
    (\email{andy.wan@mcgill.ca}).}
    \and
  Jean-Christophe Nave\thanks{Department of Mathematics and Statistics, McGill University, Montr\'eal,  QC, H3A 0B9, Canada (\email{jcnave@math.mcgill.ca}).}
}

\usepackage{amsopn}

\usepackage{amstext,amsmath,amssymb}
\usepackage{enumerate}
\newtheorem{example}{Example}
\newsiamthm{claim}{Claim}
\newsiamthm{remark}{Remark}
\newsiamthm{hypo}{Hypothesis}

\def\[#1\]{\begin{align*}#1\end{align*}}

\newcommand\norm[1]{\left\lVert#1\right\rVert}
\newcommand\cont[1]{[#1]}

\newcommand\contdep[2]{[#1]_{#2}}
\newcommand\discdep[2]{\{#1\}_{#2}}

\newcommand{\bb}{\boldsymbol }
\ifpdf
\hypersetup{
  pdftitle={\TheTitle},
  pdfauthor={\TheAuthors}
}
\fi

\begin{document}

\maketitle

% REQUIRED
\begin{abstract}
We show the arbitrarily long-term stability of conservative methods for autonomous ODEs. Given a system of autonomous ODEs with conserved quantities, if the preimage of the conserved quantities possesses a bounded locally finite neighborhood, then the global error of any conservative method with the uniformly bounded displacement property is bounded for all time, when the uniform time step is taken sufficiently small. On finite precision machines, the global error still remains bounded and independent of time until some arbitrarily large time determined by machine precision and tolerance. The main result is proved using elementary topological properties for discretized conserved quantities which are equicontinuous. In particular, long-term stability is also shown using an averaging identity when the discretized conserved quantities do not explicitly depend on time steps. Numerical results are presented to illustrate the long-term stability result. 
\end{abstract}

% REQUIRED

% REQUIRED

%********************

\section{Introduction}
In recent years, there has been vast renewed interests in structure-preserving discretizations; that is numerical methods which preserve underlying structures of differential equations at the discrete level \cite{MarWes01, Olv01, Hir03, leim04, HaiLubWan06, boch06, ArnFalWin06, ChrMunOwr11}. One primary motivation for these discretizations is, for some class of problems, the ability to preserve certain features inherent to the continuous problem is a determining factor for acceptance of numerical results. For instance, for ODEs with a Hamiltonian structure, preservation of phase space volume is a desirable feature for the discrete flow of symplectic methods \cite{HaiLubWan06}. For systems arising from variational formulation, the variational principle is preserved by variational integrators via extremizing the action integral over a finite dimensional discrete space \cite{MarWes01}. Beyond this primary motivation, structure-preserving discretizations can also possess additional stability and long-term properties. For example, symplectic methods have been shown to possess favorable long-term properties, such as near conservation of energy over an exponentially long time \cite{BenGio94}. Moreover, for completely integrable Hamiltonian systems, symplectic methods nearly conserve all first integrals depending only on action variables and have at most linear growth in the global error over an exponentially long time \cite{CalHai95,CalSan93,HaiLubWan06}. 

In the present work, we focus on the question pertaining to stability properties on the class of \emph{conservative methods}; specifically discretizations which exactly preserve conserved quantities at the discrete level. Conservative methods for ODEs and PDEs have a long history in numerical analysis \cite{CouFriLew67,LaBGre75, SimTarWon92, LiVu95, Gon96, QuiMcL08, FurMat10, DahOwr11, DahOwrYag11, CelAE12}. Traditionally, conservative methods have been proposed for various types of ODEs with special forms of conserved quantities \cite{Sha86, coo87, KanZai95, MclQui04, Cel09, BruIavTri15}.

To the best knowledge of the authors, there are two general classes of exactly conservative methods for ODEs. One is the traditional projection method where the main idea is to project the discrete solution back onto the level set of the invariants after advancing some number of time steps with a standard numerical method \cite{HaiLubWan06}. The other general conservative method is called the discrete gradient method \cite{Qui97, McL99} which is based on rewriting the ODE system in a skew-gradient form so that first integrals can be conserved discretely. More recently, another general conservative method, called the multiplier method \cite{WanBihNav17a}, has been proposed to systematically discretize ODEs using generalizations of integrating factors. 

Motivated by the general applicability of these conservatives method, it is important, in our view, to provide a general stability result for conservative methods; specifically for the case of ODEs. On the outset, this may seem like an impossible task as most discretizations constructed by these methods do not possess an a priori common structure and are often nonlinear in nature. Fortunately, these difficulties can be resolved when one takes the point of view that \emph{long-term stability} is intimately connected with the topology of the conserved quantities. Specifically for ODEs, under some appropriate conditions, we will show using basic topological arguments that the global error of a conservative method is \emph{bounded for all time}.

This paper is organized as follows. In Section 2, we review basic properties of conservative methods. It is shown that equicontinuity of discretized conserved quantities plays a central role in many estimates used in subsequent sections. In Section 3, we review some elementary topology results relevant to our current discussion and show a key separation theorem which form the basis for the main stability result. Moreover, we discuss the practical limitation of the main result in finite precision arithmetic. In Section 4, we verify numerically the main stability result for various (nonlinear) multistep methods and make comparison with traditional and symplectic methods. Finally, in the appendix, we show the uniformly bounded displacement property for a conservative 1-step method from Section 4.

%*****************
\section{Preliminaries}
\label{sec:prelim}
\subsection{Definitions and notations}
Let $n \in \mathbb{N}$ and $U$ be an open subset of $\mathbb{R}^n$. Suppose $\bb f: U\rightarrow \mathbb{R}^n$ is locally Lipschitz continuous. Then by Picard's theorem, for any $\bb x_0\in U$, there exists an open interval $I = (-T,T)$ such that the autonomous ODE,
\begin{align}
\bb F\contdep{\bb x}{t} := \dot{\bb x}(t) - \bb f(\bb x(t)) &= \bb 0\label{eq:ODE} ,\\
\bb x(0) &= \bb x_0, \nonumber
\end{align}
has an unique solution $\bb x \in C^1(I;U)$. For brevity, we used the notation $\contdep{\bb x}{t}$ denoting dependence on $t, \bb x(t)$ and higher derivatives of $\bb x(t)$.

For $1\leq m \leq n$. We assume the ODE \eqref{eq:ODE} has $m$ conserved quantities; that is there exists a continuous vector-valued function $\bb \psi: U\subset\mathbb{R}^n\rightarrow \mathbb{R}^m$ such that for the constant $\bb c = \bb \psi(\bb x_0)$, the unique solution $\bb x \in C^1(I;U)$ satisfies for $t\in I$,
\begin{equation}
\bb \psi(\bb x(t)) = \bb c. \label{contCons}
\end{equation}% and $D\bb \psi_{\bb x}$ has maximal rank for $\bb x\in \bb \psi^{-1}(\{\bb c\})$ where $\bb \psi^{-1}(\{\bb c\})$ is the preimage at $\bb c$,
%$$
%\bb \psi^{-1}(\{\bb c\}) = \{\bb x\in \mathbb{R}^n : \bb \psi(\bb x) = \bb c\}.
%$$
As to be discussed later, the preimage of $\bb c$, denoted as $\bb \psi^{-1}(\{\bb c\})$, can be written as,
\begin{equation*}
\bb \psi^{-1}(\{\bb c\}) = \bigcup_{j\in J} X_{j}, 
\end{equation*} for some countable\footnote{The number of connected components of any subsets in $\mathbb{R}^n$ is at most countable; See \cite{Mun00} or Section \ref{sec:stability}.} index set $J$ with each $X_j$ as a connected component of $\bb \psi^{-1}(\{\bb c\})$.
Denote the connected component containing $\bb x_0$ as $X_0$. If $X_0$ is compact, then by standard theory of ODEs, the local solution $\bb x$ can be extended to a global solution for all $t\in\mathbb{R}$. 

\begin{theorem}
Suppose the connected component $X_0$ containing $\bb x_0$ is compact, then the unique solution $\bb x \in C^1(I;U)$ to \eqref{eq:ODE} can be extended for all $t\in\mathbb{R}$ and so $\bb x\in C^1(\mathbb{R};X_0)$. \label{boundedness}
\end{theorem}

Let $\tau>0$ be a time step size and consider the set of uniform time steps of $\{t_k=k\tau:k\in\mathbb{N}\}$. For a given $\mu\in\mathbb{N}$, we shall consider general $\mu$-step discretizations or $\mu$-step methods (we use both terms interchangeably). In particular, let $\bb F^\tau: U\times\cdots\times U\subset\mathbb{R}^{n(\mu+1)} \rightarrow \mathbb{R}^n$ be a continuous vector-valued function depending on $\tau$. Then the $\mu$-step discretization is given by, 
\begin{align}
\bb F^\tau\discdep{\bb x^\tau}{k} :=\bb F^\tau(\bb x_{k+1},\bb x_{k},\dots,\bb x_{k-\mu+1}) = 0, \label{eq:disc}
\end{align}
Similar to the continuous case, we employ the notation $\discdep{\bb x^\tau}{k}$ to denote dependence on the successive approximation $\bb x_{k}$ at different time steps. Moreover, we shall consider discretizations which have the \emph{uniformly bounded displacement} property.

\begin{definition}
\label{def:LCP}
Let $K\subset \mathbb{R}^n$ be a compact subset and $r>0$. A $\mu$-step discretization \eqref{eq:disc} is said to have the {\bf uniformly bounded displacement} (UBD) property on $K$ with displacement $r$ if there exists $\tau_{c}>0$ depending only on $K$ and $r$ such that if $0<\tau<\tau_c$ and $\{\bb x_k, \dots, \bb x_{k-\mu+1}\} \subset K$ for each $k\geq \mu-1$, then the discretization \eqref{eq:disc} have an unique solution $\bb x_{k+1} \in \bigcap_{i=0}^{\mu-1} \overline{B_r(\bb x_{k-i})}$. A $\mu$-step discretization is said to have the UBD property if it has the UBD property for all compact $K\subset \mathbb{R}^n$ and $r>0$.
\end{definition}

Intuitively, for small enough time step $\tau$, a discretization with the UBD property always yields a solution $\bb x_{k+1}$ which cannot grow arbitrarily far away from the previous solutions $\{\bb x_k, \dots, \bb x_{k-\mu+1}\}$ for all $k\geq \mu-1$; that is their respective \emph{displacements are uniformly bounded}. 

It will be seen later that the UBD property plays an important role for showing long-term stability of conservative methods. 

\begin{definition}
A $\mu$-step discretization $\bb F^\tau$ is consistent to order $p>0$ with $\bb F$ if for $\bb x \in C^{p+1}(I;U)$ and each time step $t_k$, there exists a positive constant $C_F$ independent of $\tau$ such that,
\begin{equation*}
\norm{\bb F\contdep{\bb x}{t_k}-\bb F^\tau(\bb x(t_{k+1}), \bb x(t_{k}),\dots,\bb x(t_{k-\mu+1}))} \leq C_F(\norm{\bb x}_{C^{p+1}(\tilde{I}_k)})\tau^p,
\end{equation*}where $\tilde{I}_k = [t_{k-\mu+1},t_{k+1}]$ and $\displaystyle \norm{\bb x}_{C^{p+1}(\tilde{I_k})} = \max_{0\leq i\leq p+1} \norm{\frac{d^i \bb x}{dt^i}}_{L^\infty(\tilde{I}_k)}$.
\end{definition}
In practice, $C_F$ typically arises from Taylor expansion with remainder terms. Similarly, let $\bb\psi^\tau:U\times\cdots\times U\subset\mathbb{R}^{n\mu} \rightarrow \mathbb{R}^m$ be a continuous vector-valued function depending on $\tau$. 

\begin{definition}
The discrete conserved quantities $\bb \psi^\tau$ is consistent to order $p$ with $\bb \psi$ if for $\bb x \in C^{p}(I;U)$ and each time step $t_k$, there exists a positive constant $C_\psi$ independent of $\tau$ such that,
\begin{equation*}
\norm{\bb \psi(\bb x(t_k))-\bb \psi^\tau(\bb x(t_{k}),\dots,\bb x(t_{k-\mu+1}))} \leq C_\psi(\norm{\bb x}_{C^{p}(I_k)})\tau^p,
\end{equation*}where $I_k = [t_{k-\mu+1},t_{k}]$ and $\displaystyle \norm{\bb x}_{C^{p}(I_k)} = \max_{0\leq i\leq p} \norm{\frac{d^i \bb x}{dt^i}}_{L^\infty(I_k)}$
\end{definition}
\subsection{Equicontinuity and averaging identity}
If a family of discretized conserved quantities $\{\bb \psi^\tau\}_{0<\tau<\tau_0}$ is equicontinuous for some $\tau_0>0$, one immediate consequence is the following important relation between $\bb \psi$ and $\bb \psi^\tau$.
\begin{lemma}
Suppose $\bb \psi^\tau:U\times\cdots\times U\subset\mathbb{R}^{n \mu}\rightarrow \mathbb{R}^m$ is consistent to order $p$ with $\bb \psi$ and for some $\tau_0>0$, the family of functions $\{\bb \psi^\tau\}_{0<\tau<\tau_0}$ is equicontinuous. Then for any $\bb y\in U$,
\begin{equation*}
\bb \psi(\bb y) = \lim_{\tau\rightarrow 0}\bb \psi^\tau(\bb y,\dots, \bb y). \label{limitIdentity}
\end{equation*} \end{lemma}
\begin{proof}
For any $\epsilon>0$ and $\bb y\in U$, pick a function $\bb x(t) \in C^{(p)}([t_{k-\mu+1},t_{k}])$ with $\bb x(t_k)=\bb y$. For fixed $t_k$, by continuity of $\bb x(t)$, there exists a positive constant $\tau_1$ such that if $0<\tau<\tau_1$, then $\bb x(t_{k-i})\in U$ for $i=0,\dots, \mu-1$. Thus, by equicontinuity and since all norms are equivalent on $\mathbb{R}^{n \mu}$, there exists a positive constant $\delta$ depending only on $\epsilon$ such that if $\displaystyle \max_{0\leq i\leq \mu-1} \norm{\bb y-\bb x(t_{k-i})} < \delta$,
\begin{equation*}
\norm{\bb \psi^\tau (\bb x(t_k), \bb x(t_{k-1}),\dots, \bb x(t_{k-\mu+1})) - \bb \psi^\tau(\bb y, \bb y,\dots, \bb y)}\leq \frac{\epsilon}{2},
\end{equation*} for all $0<\tau<\min\{\tau_0,\tau_1\}$. Moreover, by continuity of $\bb x$ again, there exists some positive constant $\tau_2$ such that if $0<\tau<\tau_2$, $\norm{\bb y-\bb x(t_{k-i})} < \delta$ for all $i=0,\dots,\mu-1$.
Combining together with consistency, this implies for $\displaystyle 0<\tau<\min\left\{\tau_0,\tau_1,\tau_2,\left(\frac{\epsilon}{2 C_\psi(\norm{\bb x}_{C^{p}(I_k)})}\right)^{\frac{1}{p}}\right\}$,
\begin{align*}
\norm{\bb \psi(\bb y) - \bb\psi^\tau(\bb y,\dots,\bb y)} & \leq \norm{\bb \psi(\bb x(t_k)) - \bb \psi^\tau (\bb x(t_k), \bb x(t_{k-1}),\dots, \bb x(t_{k-\mu+1}))}\\
&\hskip 3mm + \norm{\bb \psi^\tau (\bb x(t_k), \bb x(t_{k-1}),\dots, \bb x(t_{k-\mu+1})) - \bb\psi^\tau(\bb y,\bb y,\dots,\bb y)}\\
&\leq C_\psi(\norm{\bb x}_{C^{p}(I_k)}) \tau^p + \frac{\epsilon}{2} < \epsilon.
\end{align*}
In other words, the limit $\bb \psi^\tau(\bb y,\dots, \bb y)$ as $\tau\rightarrow 0$ exists and is equal to $\bb \psi(\bb y)$.
\end{proof}
If $\bb \psi^\tau$ does not depend on $\tau$ explicitly, then Theorem \ref{limitIdentity} follows immediately, as $\bb \psi^\tau$ is trivially equicontinuous on $U$. In fact, the following remarkably simple identity holds.
\begin{corollary}[Averaging Identity]
Suppose $\bb \psi^\tau:U\times\cdots\times U\subset\mathbb{R}^{n \mu}\rightarrow \mathbb{R}^m$ is consistent to order $p$ with $\bb \psi$ and assume $\bb \psi^\tau$ does not depend on $\tau$ explicitly. Then for any $\bb y\in \mathbb{R}^n$, \label{avgIdentity}
\begin{equation*}
\bb \psi(\bb y) = \bb \psi^\tau(\bb y,\dots, \bb y).
\end{equation*} 
\end{corollary}
\begin{proof}
Let $\bb y\in \mathbb{R}^n$ and $t_k$ be fixed and let $\bb x(t) \in C^{(p)}([t_{k-\mu+1},t_{k}])$ with $\bb x(t_k)=\bb y$. By continuity of $\bb x$ and for fixed $t_k$, $\lim_{\tau\rightarrow 0} \bb x(t_{k-i})=\bb y$ for all $i = 0 ,\dots \mu-1$. Since $\bb \psi^\tau$ does not depend on $\tau$ explicitly, then by consistency and continuity of $\bb \psi^\tau$,
\[
&\norm{\bb \psi(\bb y)-\bb \psi^\tau (\bb y, \bb y,\dots, \bb y)} = \lim_{\tau\rightarrow 0} \norm{\bb \psi(\bb y)-\bb \psi^\tau (\bb y, \bb y,\dots, \bb y)}\\
&\hskip 4mm\leq  \lim_{\tau\rightarrow 0}\norm{\bb \psi(\bb y) - \bb \psi^\tau (\bb y, \bb x(t_{k-1}),\dots, \bb x(t_{k-\mu+1}))}\\
&\hskip 8mm +  \lim_{\tau\rightarrow 0}\norm{\bb \psi^\tau (\bb y, \bb x(t_{k-1}),\dots, \bb x(t_{k-\mu+1})) - \bb\psi^\tau(\bb y,\bb y,\dots,\bb y)}\\
&\hskip 4mm\leq  \underbrace{\lim_{\tau\rightarrow 0} C_{\psi}(\norm{\bb x}_{C^{p}(I_k)}) \tau^p}_{=0}\\
&\hskip 8mm + \underbrace{\norm{\bb \psi^\tau (\bb y, \lim_{\tau\rightarrow 0} \bb x(t_{k-1}),\dots, \lim_{\tau\rightarrow 0} \bb x(t_{k-\mu+1})) - \bb\psi^\tau(\bb y,\bb y,\dots,\bb y)}}_{=0}
\]
\end{proof}
To the best knowledge of the authors, we have not seen this remarkably simple identity relating $\bb \psi$ and $\bb \psi^\tau$ appeared in the previous literature. The term \emph{averaging identity} originates from applications where such $\bb \psi^\tau$ can typically be interpreted as a kind of (nonlinear) average of $\bb \psi$ among different time steps $t_k$.

\subsection{Conservative discretization}

\begin{definition}
The discretization \eqref{eq:disc} is called conservative if
\begin{equation*}
\bb \psi^\tau\discdep{\bb x^\tau}{k+1} = \bb \psi^\tau\discdep{\bb x^\tau}{k}, \text{ for } k\in\{\mu-1,\mu,\dots\}.
\end{equation*}
\end{definition}For a conservative discretization, it follows by induction that for $k\in\{\mu-1,\mu,\dots\}$,
\begin{equation}\bb \psi^\tau\discdep{\bb x^\tau}{k} =\bb \psi^\tau(\bb x_{k},\dots,\bb x_{k-\mu}) = \bb \psi^\tau(\bb x_{\mu-1},\dots,\bb x_{0}) =:\bb c^\tau.\label{discCons}
\end{equation}

In the case of 1-step methods with $\bb \psi^\tau$ not explicitly depending on $\tau$, then the averaging identity of Lemma \ref{limitIdentity} implies $\bb \psi^\tau(\bb y) = \bb \psi(\bb y)$ and so $\bb c^\tau = \bb c$. For general $\mu$-step methods, $\mu-1$ initial values must be specified in order to proceed. This initialization step is usually handled by using one-step methods of sufficient order, such as Runge-Kutta methods. However, since traditional 1-step methods are generally not conservative, there will be a corresponding error in the constant $\bb c^\tau=\bb \psi^\tau(\bb x_{\mu-1},\dots,\bb x_{0})$. Fortunately, as we show in Section \ref{sec:stability}, this initialization error does not pose a problem for the long-term stability result, as long as we can choose the error in $\norm{\bb c-\bb c^\tau}$ to be arbitrarily small. In particular, we need the following result in subsequent section.

\begin{lemma}
Let $\bb \psi^\tau:U\times\cdots\times U\subset\mathbb{R}^{n\mu} \rightarrow \mathbb{R}^m$ be consistent to order $p$ with $\bb \psi$. Suppose the $\mu$ initial values are $p$-th order accurate; that is for the unique solution $\bb x\in C^1(I;U)\cap C^p([0,t_{\mu-1}];U)$ to the ODE \eqref{eq:ODE}, the $\mu$ initial values $\{\bb x_k\}_{k=0}^{\mu-1}$ satisfies for some positive constants $C, \tau_{\mu}$ independent of $\tau$ such that if $0<\tau<\tau_{\mu}$,
\begin{equation}
\max_{0\leq k\leq \mu-1} \norm{\bb x(t_k)-\bb x_k} \leq C\tau^p. \label{initHypo}
\end{equation}
Also assume for some $\tau_0>0$, the family of functions $\{\bb \psi^\tau\}_{0<\tau<\tau_0}$ is equicontinuous on $U\times\cdots\times U \subset\mathbb{R}^{n \mu}$. 
Then for $\bb c=\bb \psi(\bb x_0)$,  \label{initLemma}
\begin{equation*}\lim_{\tau\rightarrow 0}\norm{\bb c-\bb c^\tau}=0.\end{equation*}
\end{lemma}
\begin{proof} The proof is similar to Lemma \ref{limitIdentity}. Let $\epsilon>0$ and $\bb x$ be the exact solution to the ODE \eqref{eq:ODE} and $\{\bb x_k\}_{k=0}^{\mu-1}$ be the given initial values. From \eqref{initHypo} and that $\bb x(t_k)\in U$ for all $0\leq k\leq \mu-1$, it follows that for some positive constant $\tau_1$, $\bb x_k \in U$ for all $0\leq k\leq \mu-1$ and $0<\tau<\tau_1$. By equicontinuity, there exists a $\delta>0$ depending only on $\epsilon$ such that if $\max_{0\leq k\leq \mu-1}\norm{\bb x(t_k)-\bb x_k}<\delta$, then
\begin{equation*}
\norm{\bb \psi^\tau(\bb x(t_{\mu-1}),\dots,\bb x(t_{0}))-\bb \psi^\tau(\bb x_{\mu-1},\dots,\bb x_{0})} < \frac{\epsilon}{2},
\end{equation*} for all $0<\tau<\min\{\tau_0,\tau_1\}$.
Indeed, $\max_{0\leq k\leq \mu-1}\norm{\bb x(t_k)-\bb x_k}<\delta$ is fulfilled by hypothesis \eqref{initHypo} if $0<\tau<\tau_2$ for some positive constant $\tau_2$. Since $\bb c = \bb\psi(\bb x(t_{\mu-1}))$ by \eqref{contCons} and $\bb c^\tau =\bb \psi^\tau(\bb x_{\mu-1},\dots,\bb x_{0})$ by \eqref{discCons}, it follows from consistency that for sufficiently small $\tau$, 
\begin{align*}
\norm{\bb c-\bb c^\tau} 
& \leq \norm{\bb\psi(\bb x(t_{\mu-1})) - \bb \psi^\tau(\bb x(t_{\mu-1}),\dots,\bb x(t_{0}))}\\
&\hskip 4mm+\norm{\bb \psi^\tau(\bb x(t_{\mu-1}),\dots,\bb x(t_{0}))-\bb \psi^\tau(\bb x_{\mu-1},\dots,\bb x_{0})} \\
&\leq C_\psi(\norm{\bb x}_{C^{p}(I_k)})\tau^p + \frac{\epsilon}{2} < \epsilon.
\end{align*}
\end{proof}

\section{Main results}
\label{sec:stability}

We now discuss a long-term stability result for conservative methods. Although in application, we have in mind the underlying space is $X=\mathbb{R}^n$, which is sufficient for ODEs in finite dimensions. In anticipation for subsequent work on evolution PDEs, $X$ can be some function space where the PDEs are viewed as ODEs over infinite dimensional spaces. Since the main ideas are mostly based on topological properties, we will state the main theorem in a general setting and restricting $X$ to a metric space when necessary. First, we review some elementary results from topology relevant to our discussion. See \cite{Mun00} for more details.

%\footnote{Recall, a topological space is normal if for any two disjoint closed subsets $U$ and $V$, there exists two disjoint open subsets $A$ and $B$ such that $U\subset A$ and $V\subset B$.}

\begin{theorem} 
Let $A\subset X$ be a nonempty subset of a locally connected, second-countable topological space $X$. Then $A=\bigcup_{j\in J} A_j$ for some countable indexed set $J$, where the collection of $A_j$ are connected components of $A$ with each $A_j$ being nonempty, closed in $A$ and disjoint from each other.\label{countableComp}
\end{theorem}

Theorem \ref{countableComp} immediately implies the following:
\begin{lemma} \label{countableXj}
Let $X$ and $Y$ be topological spaces with $X$ locally connected and second-countable and $Y$ Hausdorff. Suppose $\bb \psi:X\rightarrow Y$ is a continuous function. For any $\bb c\in Y$ with a nonempty preimage $\bb \psi^{-1}(\{\bb c\})$, there is some countable indexed set $J$ such that,
\begin{equation*}
\bb \psi^{-1}(\{\bb c\}) = \bigcup_{j\in J} X_{j},
\end{equation*} where each $X_{j}$ is a nonempty, closed subset in $\bb \psi^{-1}(\{\bb c\})$ and disjoint from each other. 
\end{lemma}

Similarly, we will be interested in working with preimage of neighborhoods in metric spaces. Specifically, for a metric space $Y$ with a metric $d_Y(\cdot,\cdot)$, we denote 
the open neighborhood $B_{\epsilon}(\bb c) = \{ \bb y \in Y : d_Y(\bb y, \bb c) <\epsilon \}$. 

\begin{lemma} \label{countableXje}
Let $X$ be a locally connected, second-countable topological space and $Y$ be a metric space. Suppose $\bb \psi:X\rightarrow Y$ is a continuous function. For any $\bb c\in Y$ and any $\epsilon>0$ with a nonempty preimage $\bb \psi^{-1}(\overline{B_{\epsilon}(\bb c)})$, there is some countable indexed set $J^\epsilon$ such that,
\begin{equation*}
\bb \psi^{-1}(\overline{B_{\epsilon}(\bb c)}) = \bigcup_{j\in J^\epsilon} X_{j}^\epsilon,
\end{equation*} where each $X_{j}^\epsilon$ is a nonempty, closed subset in $\bb \psi^{-1}(\overline{B_{\epsilon}(\bb c)})$ and disjoint from each other. 
\end{lemma}

\begin{lemma}%[Theorem 17.3, \cite{Mun00}] 
Let $X$ be a topological space. If $A\subset X$ is closed in $X$ and $B\subset A$ is closed in $A$, then $B$ is closed in $X$. \label{closedSubspace}
\end{lemma}

\begin{lemma}
If the hypotheses of Lemma \ref{countableXj} are satisfied, then each $X_{j}$ is a closed subset of $X$. Moreover, if $Y$ is a metric space, then each $X_j^\epsilon$ is closed in $X$ for any $\epsilon>0$. 
\label{closedLem}
\end{lemma}
\begin{proof}
By continuity of $\bb \psi$ and $\{\bb c\}$ is closed in $Y$ (since $Y$ is Hausdorff), $\bb \psi^{-1}(\{\bb c\})$ is closed in $X$. Since $X_{j}$ is closed in $\bb \psi^{-1}(\{\bb c\})$, then Lemma \ref{closedSubspace} implies $X_{j}$ is closed in $X$. The proof proceeds similarly for $X_j^\epsilon$.
\end{proof}

\subsection{Locally finite neighborhood}
For the main stability result, we wish to include cases where the connected components can be bounded or unbounded. Moreover, we also wish to handle the possibility of countably infinitely many connected components. However, it turns out the main stability result hinges on whether certain connected components can be separated by open neighborhoods. In particular, there are situations which can arise we wish to exclude, as the following example illustrate.

\begin{example}Let $X=\mathbb{R}=Y$ and consider the smooth function: 
\begin{equation*}
\psi(x)=\left\{\begin{split}
\exp\left(-\frac{1}{x^2}\right)\sin\left(\frac{1}{x^2}\right),&& x\neq 0 \\
0,&& x=0
\end{split}\right.
\end{equation*}The preimage $\psi^{-1}(\{0\})$ has the connected components $\{0\}\cup \bigcup_{k\in\mathbb{N}}\left\{\pm\frac{1}{\sqrt{k\pi}}\right\}$. But for $\epsilon>0$, the neighborhood $(-\epsilon,\epsilon)$ around $X_0=\{0\}$ intersects infinite many connected components $\left\{\pm\frac{1}{\sqrt{k\pi}}\right\}$ for all $k>\frac{1}{\epsilon^2\pi}$. 
\end{example}
The preceding example shows that $X_0$ cannot be separated by open neighborhoods if it has ``too many" neighboring connected components. This leads to the following definition.

\begin{definition}
Let $\{ U_\beta\}_{\beta \in J}$ be a collection of closed subsets of a topological space $X$ with an index set $J$. For a fixed $\alpha\in J$, $U_\alpha$ is said to have a {\bf locally finite neighborhood} (LFN) $V$ if $V$ is an open subset of $X$ such that:
\begin{itemize}
\item $U_\alpha \subset V$
\item $U_\beta \cap V = \varnothing$ for all but finitely many $\beta \in J$ 
\end{itemize}
Furthermore, if $X$ is a metric space, we say that $V$ is a bounded LFN of $U_\alpha$ if $V$ is also bounded.
\end{definition}

For a normal topological space $X$, an equivalent definition of a LFN is that $U_{\alpha}$ and $\bigcup_{\beta \neq \alpha} U_{\beta}$ can be separated by open neighborhoods.

\begin{theorem}Let $X$ be a normal topological space and $\{ U_\beta \}_{\beta \in J}$ be a collection of closed subsets in $X$. Then $U_{\alpha}$ has a LFN $V$ if and only if there exists disjoint open subsets $A, B$ in $X$ such that $U_{\alpha}\subset A\subset V$ and $\bigcup_{\alpha\neq \beta \in J} U_{\beta} \subset B$. \label{LFN}
\end{theorem}
\begin{proof}
It suffices to prove only the forward implication, since if there exists such disjoint open subsets $A, B$, then $U_{\alpha}$ has a LFN $A$. Suppose $U_{\alpha}$ has a LFN $V$ so that $U_{\alpha}\subset V$ and at most a finite collection $\{U_{\beta}\}_{\beta\in J'}$ with $\alpha \notin J'$ such that $U_{\beta}\cap V \neq \varnothing$ for all $\beta\in J'$. Since $J'$ is finite, $\bigcup_{\beta\in J'} (U_{\beta}\cap \overline{V})$ is closed in $X$. Since $U_{\alpha}$ is closed in a normal topological space $X$, there exists disjoint open subsets $A', B'$ in $X$ such that $U_{\alpha}\subset A'$ and $\bigcup_{\beta \in J'}(U_{\beta}\cap \overline{V}) \subset B'$. Let $A = A'\cap V$ and $B = B' \bigcup (X-\overline{V})$. Then clearly both $A, B$ are disjoint and open in $X$ with $U_{\alpha}\subset A\subset V$. Thus, the result follows since, 
\[
\bigcup_{\alpha \neq\beta\in J} U_{\beta} = \underbrace{\left(\bigcup_{\beta \in J'} (U_{\beta}\cap \overline{V})\right)}_{\subset B'} \bigcup \underbrace{\left(\bigcup_{\beta\in J'} (U_{\beta}\cap (X-\overline{V})\right)}_{\subset X-\overline{V}} \bigcup \underbrace{\left(\bigcup_{\substack{\beta\in J\\ \beta \notin J'}}^\infty U_{\beta}\right)}_{\subset X-\overline{V}} \subset B.
\]
\end{proof}

\begin{corollary}
Let $X$ be a metric space and $Y$ be a Hausdorff topological space. Suppose $\bb \psi:X\rightarrow Y$ is continuous function with a nonempty preimage $\bb \psi^{-1}(\{\bb c\}) = \bigcup_{j\in J} X_j$ for some countable index set $J$. Then $X_0$ has a LFN $V$ if and only if there exists disjoint open subsets $A, B$ in $X$ such that $X_0\subset A\subset V$ and $\bigcup_{0\neq j \in J}X_{j} \subset B$. Thus, if $X_0$ has a bounded LFN $V$, then $A$ is also bounded. \label{LFNms}
\end{corollary}

\begin{proof}
Since any metric space $X$ is locally connected and second-countable, $X_{j}$ is closed in $X$ by Lemma \ref{closedLem} for all $j\in J$. As $X$ is also normal, applying Theorem \ref{LFN} for the collection of closed subsets $\{X_{j}\}_{j\in J}$ implies the result.
\end{proof}

Furthermore, we will need the following two lemmas regarding compact subsets.

\begin{lemma}
Let $X$ be a topological space and let $\{A_n\}_{n\in \mathbb{N}}$ be a decreasing nested sequence of nonempty compact subsets in $X$. For any open subset $U$ in $X$ such that $\displaystyle \bigcap_{n\in\mathbb{N}} A_n \subset U$, there exists a positive integer $N$ such that $A_n\subset U$ for all $n\geq N$. \label{nestedCompactSubsets}
\end{lemma}
\begin{proof}
If not, then there exists a sequence $n_i\rightarrow \infty$ such that $A_{n_i}\cap (X-U)\neq \varnothing$ for all $i\in \mathbb{N}$. Since each $B_i:=A_{n_i}\cap (X-U)$ is nonempty and compact, there exists an element $\bb x\in \bigcap_{i\in\mathbb{N}}B_i$ by Cantor's intersection theorem. Moreover, for each $n\in\mathbb{N}$, there exists an index $n_{i_j}\geq n$ so that $A_{n_{i_j}}\subset A_n$, since $A_n$ are decreasing and $n_i\rightarrow \infty$. This implies $\bigcap_{j\in\mathbb{N}} A_{n_{i_j}} \subset \bigcap_{n\in\mathbb{N}} A_n$. However, this leads to a contradiction since this would imply $\bb x \in B_i \subset X-U$ and $\bb x \in \bigcap_{i\in\mathbb{N}} B_i \subset \bigcap_{i\in\mathbb{N}} A_{n_i} \subset \bigcap_{j\in\mathbb{N}} A_{n_{i_j}}  \subset \bigcap_{n\in\mathbb{N}} A_{n} \subset U$.

\end{proof}

\begin{lemma}
Let $X$ be a metric space and $A, B$ be subsets of $X$ with $A$ compact in $X$ and $A\cap \overline{B} = \varnothing$. Then $d_X(A,B)>0$. \label{compactDist}
\end{lemma}
\begin{proof}
Suppose not, then there is a sequence $a_n\in A$ such that $d_X(a_n,B)\rightarrow 0$. Since $A$ is compact, there is a convergent subsequence $a_{n_i}\rightarrow a \in A$. 
Thus, $d_X(a,B)=\lim_{i\rightarrow \infty} d_X(a_{n_i},B)=0$ or in other words $a\in \overline{B}$ which contradicts that $A\cap \overline{B} = \varnothing$.
\end{proof}

Finally, we show a key separation theorem for establishing the main stability theorem for conservative methods. Note that for any $\epsilon>0$, since $X_0 \subset \bb \psi^{-1}(\overline{B_\epsilon(\bb c)})$, $X_0\subset X_j^\epsilon$ for some index $j$ in $J^\epsilon$. By rearranging $J^\epsilon$ if necessary, we can denote $X_{0}^\epsilon$ to be the unique connected component containing $X_0$ for all $\epsilon>0$.

\begin{theorem}
Let $X, Y$ be metric spaces and let $\bb \psi:X\rightarrow Y$ be a continuous function with a nonempty preimage $\bb \psi^{-1}(\{\bb c\}) = \bigcup_{j\in J} X_j$ for some countable index set $J$. For each $\epsilon>0$, denote the nonempty preimage $\bb \psi^{-1}(\overline{B_\epsilon(\bb c)}) = \bigcup_{j\in J^\epsilon} X_j^\epsilon$ for some countable index set $J^\epsilon$. Suppose $X_0$ has a bounded LFN $V$ with $\overline{V}$ compact, then there exists $\epsilon_0>0$ such that if $0<\epsilon<\epsilon_0$, $X_{0}^\epsilon$ is compact and is separated from $\bigcup_{0\neq j \in J^\epsilon} X_j^\epsilon$. \label{keyLemma}
\end{theorem}
\begin{proof}
Let $A \subset V$ and $B$ be such disjoint open sets from Corollary \ref{LFNms}. Now suppose for all $\epsilon>0$, there exists $\bb x \in \bb \psi^{-1}(\overline{B_{\epsilon}(\bb c)}) \cap (X-(A\bigcup B))$. Then for all $\epsilon>0$, $\bb x \in \psi^{-1}(\overline{B_{\epsilon}(\bb c)})$, or equivalently $\bb \psi(\bb x)=\bb c$. This implies $\bb x \in X_{j} \subset A\bigcup B$ for some $j\in J$, which contradicts $\bb x \in X-(A\bigcup B)$. It follows that there exists $\epsilon'>0$ so that if $0<\epsilon\leq\epsilon'$, 
\[
\bigcup_{j \in J^\epsilon} X_j^\epsilon = \bb \psi^{-1}(\overline{B_{\epsilon}(\bb c)}) \subset A\bigcup B.
\]
Since $A, B$ are disjoint and $X_j^{\epsilon'}$ is connected for any $j\in J^{\epsilon'}$, either $X_j^{\epsilon'}\subset A$ or $X_j^{\epsilon'} \subset B$. In the case when $j=0$, then $X_0^{\epsilon'}\subset A$,
since otherwise $X_0 \subset X_0^{\epsilon'}\subset B$ which contradicts $X_0\cap B = \varnothing$. Moreover, $X_0^{\epsilon'}$ is compact, since $X_0^{\epsilon'}$ is closed by Lemma \ref{closedLem} and $X_0^{\epsilon'}\subset A \subset V\subset \overline{V}$ with $\overline{V}$ compact. Similarly for $0\neq j \in J^{\epsilon'}$, in the case if $X_j^{\epsilon'} \subset A$, then $X_j^{\epsilon'} \subset A-X_0^{\epsilon'}$, since $X_j^{\epsilon'}$ and $X_0^{\epsilon'}$ are disjoint if $j\neq 0$. Thus, for any $0<\epsilon\leq\epsilon'$,
\begin{equation}
\bigcup_{0\neq j \in J^\epsilon} X_j^\epsilon \subset \bigcup_{0\neq j \in J^{\epsilon'}} X_j^{\epsilon'} \subset (A-X_0^{\epsilon'})\bigcup B
\label{containinBp}
\end{equation}
Now define the following two disjoint open subsets,
\[
A' := int (X_0^{\epsilon'}), \hskip 5mm B' := (A-X_0^{\epsilon'})\bigcup B.
\]
For the moment, assume the following claim is true:
\begin{claim}
For some $\epsilon_0\leq\epsilon'$, $X_0^\epsilon$ is compact and $X_0^\epsilon \subset A'$ for all $0<\epsilon<\epsilon_0$.
\label{containinAp}
\end{claim}
Thus, combining \eqref{containinBp} and Claim \ref{containinAp} implies the theorem. It remains to show Claim \ref{containinAp} for which we proceed in two main steps. First, we show that,
\begin{equation}
X_0 \subset A'. \label{X0inAp}
\end{equation} Indeed, let $\bb x \in X_0\subset X_0^{\epsilon'}$, then by continuity of $\bb \psi$, there exists $r>0$ so that, 
$$B_r(\bb x)\subset \bb \psi^{-1}(B_{\epsilon'}(\bb c)) \subset \bigcup_{j \in J^{\epsilon'}} X_j^{\epsilon'}$$
Since $B_r(\bb x)$ is connected and $\{X_j^{\epsilon'}\}_{j\in J^{\epsilon'}}$ are connected components, $B_r(\bb x) \subset X_j^{\epsilon'}$ for some $j\in J^{\epsilon'}$. Supposing $j\neq 0$ implies the contradiction that $\bb x \in X_j^{\epsilon'} \cap X_0^{\epsilon'} =\varnothing$. So $B_r(\bb x) \subset X_0^{\epsilon'}$ or in other words $\bb x$ is an interior point of $X_0^{\epsilon'}$ which implies \eqref{X0inAp}.

Secondly, we show there exists $\epsilon_0\leq\epsilon'$ such that if $0<\epsilon<\epsilon_0$,
\begin{equation}
X_0^\epsilon \subset A'. \label{X0epsinAp}
\end{equation}
To show \eqref{X0epsinAp}, define $A_n:=X_0^{\frac{\epsilon'}{n}}$. Since each $A_{n+1} \subset A_n$ are closed by Lemma \ref{closedLem} and $X_0\subset A_n \subset X_0^{\epsilon'}$ with $X_0^{\epsilon'}$ compact, $\{A_n\}_{n\in\mathbb{N}}$ is a collection of nonempty, nested, compact subsets. Moreover, $\bigcap_{n\in \mathbb{N}} A_n=X_0 \subset A'$ by \eqref{X0inAp}. Thus, by Lemma \ref{nestedCompactSubsets}, there exists a positive integer $N$ such that $A_n \subset A'$ if $n\geq N$. In other words, for $\epsilon_0:= \frac{\epsilon'}{N}$, $X_0^{\epsilon} \subset A'$ if $0<\epsilon<\epsilon_0$ and each $X_0^{\epsilon}$ is compact since $A'\subset X_0^{\epsilon'}$ is compact as shown earlier. Thus, Claim \ref{containinAp} is proved.
\end{proof}

\begin{remark} The assumption that $\overline{V}$ is compact in Theorem \ref{keyLemma} can be omitted for metric spaces with the Heine-Borel property, such as when $X=\mathbb{R}^n$.
\end{remark}

\subsection{Long-term stability theorem}
We are now in the position to show the long-term stability result. In essence, under appropriate conditions, the global error is bounded for all time for conservative methods.

\begin{theorem}[Main stability theorem] \label{mainTheorem}
Let $X=\mathbb{R}^n$ and $Y=\mathbb{R}^m$ with the Euclidean norm $\norm{\cdot}$ as their metric. Let $\bb x\in C^1(I;U)\cap C^p([0,t_{\mu-1}];U)$ be the unique solution to the ODE \eqref{eq:ODE} with $\bb x(0)=\bb x_0$ and $\bb \psi(\bb x_0) = \bb c$. For some $\tau_0>0$, assume the family of functions $\{\bb \psi^\tau\}_{0<\tau<\tau_0}$ is equicontinuous on $U\times\cdots\times U \subset\mathbb{R}^{n \mu}$ and let $\bb x_{k+1}$ be the unique solution to a  conservative $p$-th order $\mu$-step discretization \eqref{eq:disc} with the UBD property where $\bb \psi^\tau(\bb x_{\mu-1},\dots,\bb x_0) = \bb c^\tau$ and the $\mu$ initial values satisfies the hypothesis of Lemma \ref{initLemma}.

If $X_0$ has a bounded LFN with $\bb x_0 \in X_0\subset U$, then there exists positive constants $\tau^*$, $C$ (independent of $\tau$ and $k$) such that for all $0<\tau<\tau^*$ and $k\in \mathbb{N}$,
\begin{equation*}
\norm{\bb x(t_k)-\bb x_{k}} \leq C.
\end{equation*}
\end{theorem}
\begin{proof} 
Since $X_0$ is closed and bounded in $X$, by Lemma \ref{closedLem}, $X_0$ is compact by Heine-Borel's theorem and a global solution $\bb x \in C^{1}(\mathbb{R};X_0)$ exists by Theorem \ref{boundedness}. 

Now the proof proceeds in three main steps: First, we will show that $X_0^\epsilon\subset U$ for sufficiently small $\epsilon$ and that $X_0^\epsilon$ is separated from the other connected components $X_j^\epsilon$ for small enough $\epsilon$. Second, we will show by induction and by the uniformly bounded displacement property that $\bb x_{k+1}$ lies in the pre-image $\bb \psi^{-1}(\overline{B_{\epsilon}(\bb c)})$ for sufficiently small $\tau$, which would imply $\bb x_{k+1}\in X_0^\epsilon$ via a contradiction argument. Finally, we combine these two main results to show the long term stability for conservative methods. 

The first main step is to show the following claim.
\begin{claim}There is some $\epsilon_1>0$ such that if $0<\epsilon<\epsilon_1$, then $X_0^\epsilon\subset U$ is compact and $d_X\left(X_{0}^\epsilon, \bigcup_{0\neq j \in J^\epsilon} X_j^\epsilon\right)>0$. \label{claim:1}
\end{claim}

To show Claim \ref{claim:1}, note that by Theorem \ref{keyLemma} and since $X_0$ has a LFN, there exists $\epsilon_0>0$ so that if $0<\epsilon<\epsilon_0$, $X_{0}^\epsilon$ is compact and is separated from $\bigcup_{0\neq j \in J^\epsilon} X_j^\epsilon$. Furthermore, define the nested nonempty compact subsets $A_n:=X_{0}^{\frac{\epsilon_0}{n}}$ and so $\bigcap_{n\in \mathbb{N}} A_n = X_0 \subset U$. Thus, by Lemma \ref{nestedCompactSubsets}, for some positive integer $N$, $X_{0}^\epsilon \subset U$ if $0<\epsilon < \epsilon_1:=\frac{\epsilon_0}{N}$. So if $0<\epsilon<\epsilon_1$, $X_{0}^\epsilon \subset U$ and Lemma \ref{compactDist} implies there is a constant $D_\epsilon>0$ so that, 
\begin{equation}
D_\epsilon\leq \norm{\bb x-\bb y}, \text{ for } \bb x \in X_{0}^\epsilon, \bb y \in \bigcup_{0\neq j \in J^\epsilon} X_j^\epsilon, \label{positiveDist}
\end{equation} which shows Claim \ref{claim:1}. In the following, $\epsilon$ is any fixed value satisfying $0<\epsilon<\epsilon_1$. 

The second main step is to show the next claim using strong induction.
\begin{claim}
There is some $\tau^*>0$ (independent of $k$) such that if $0<\tau<\tau^*$, then $\bb x_{k}\in X_0^\epsilon$ for all $k\in \mathbb{N}$. \label{claim:2}
\end{claim}

To show Claim \ref{claim:2}, we first establish the base cases of $\bb x_{k}\in X_0^\epsilon$ holds for all $0\leq k \leq \mu-1$ with sufficiently small $\tau$. 
First note that for $\delta_1>0$ small enough, the closed ball around $X_0$, $\overline{B_{\delta_1}(X_0)}$, is contained in $X_0^\epsilon$. Indeed, define the nested nonempty compact subsets $A_n:=\overline{B_{\frac{1}{n}}(X_0)}$ and so $\bigcap_{n\in \mathbb{N}} A_n = X_0 \subset X_0^\epsilon$. Thus, by Lemma \ref{nestedCompactSubsets}, for some positive integer $N$, $\overline{B_{\delta_1}(X_0)}\subset X_0^\epsilon$ if $0<\delta_1<\frac{1}{N}$. Moreover, by the hypothesis of Lemma \ref{initLemma}, there exists a constant $\tau_{\mu}>0$ such that if $0<\tau<\tau_{\mu}$, $\norm{\bb x(t_k)-\bb x_k} \leq \delta_1$ for all $0\leq k \leq \mu-1$. Since $\bb x(t) \in X_0$ for all $t \in \mathbb{R}$, then this implies that $\bb x_k\in \overline{B_{\delta_1}(X_0)} \subset X_0^\epsilon$ for all $0\leq k \leq \mu-1$, if $0<\tau<\tau_{\mu}$. This shows the base cases of Claim \ref{claim:2}.  

Next we prove the strong induction step; that is there is a $\tau'>0$ (independent of $k$) so that if $0<\tau<\tau'$ and $\{\bb x_k,\dots,\bb x_{k-\mu+1}\}\subset X_{0}^\epsilon$, then $\bb x_{k+1}\in X_{0}^\epsilon$ for all $k\geq \mu-1$. 

By the strong induction hypothesis and Claim \ref{claim:1}, $\{\bb x_k,\dots,\bb x_{k-\mu+1}\}\subset X_{0}^\epsilon \subset U$, which is in the domain of $\bb \psi$ and $\bb \psi^\tau$. However, before we can proceed to evaluate $\bb \psi$ and $\bb \psi^\tau$ at $\bb x_{k+1}$, we need to guarantee that such expressions are well-defined; that is $\bb x_{k+1}\in U$ for sufficiently small $\tau$. To show this, we again employed Lemma \ref{nestedCompactSubsets} by defining $A_n := \overline{B_{\frac{1}{n}}(X_0^\epsilon)}$ and so $\bigcap_{n\in \mathbb{N}} A_n = X_0^\epsilon \subset U$ by Claim \ref{claim:1}. Hence, for some large $N$, $\overline{B_{\frac{1}{N}}(X_0^\epsilon)} \subset U$. Now by the strong induction hypothesis that $\{\bb x_k,\dots,\bb x_{k-\mu+1}\}\subset X_{0}^\epsilon$, the UBD property implies there exists a $\tau_1>0$ such that if $0<\tau<\tau_1$, $\bb x_{k+1}\in \displaystyle \bigcap_{0\leq i\leq \mu-1} \overline{B_{\frac{1}{N}}(\bb x_{k-i})} \subset \overline{B_{\frac{1}{N}}(X_{0}^\epsilon)} \subset U$. So $\bb x_{k+1} \in U$, if $0<\tau<\tau_1$, and the expressions $\bb \psi$ and $\bb \psi^\tau$ evaluated at $\bb x_{k+1}$ are well-defined. 

Now we are in the position to derive various estimates for $\bb \psi$ and $\bb \psi^\tau$. By equicontinuity and Lemma \ref{limitIdentity}, there exists a $\tau_2>0$ depending only on $\epsilon$ such that if $0<\tau<\tau_2\leq \min\{\tau_0, \tau_{1}\}$,
\begin{equation}
\norm{\bb \psi(\bb x_{k+1})-\bb \psi^\tau(\bb x_{k+1},\bb x_{k+1},\dots,\bb x_{k+1})} \leq \frac{\epsilon}{3}.
\label{eq:est1}
\end{equation}
Also by Lemma \ref{initLemma}, there exists a $\tau_3>0$ depending only on $\epsilon$ such that if $0<\tau<\tau_3\leq \min\{\tau_0, \tau_{1}\}$,
\begin{equation}
\norm{\bb c-\bb c^\tau} \leq \frac{\epsilon}{3}.
\label{eq:est2}
\end{equation}
Moreover, by equicontinuity for $0<\tau<\tau_4\leq \min\{\tau_0, \tau_{1}\}$, there exists a $\delta_2>0$ depending only on $\epsilon$ such that if $\max_{0\leq i\leq \mu-1}\norm{\bb x_{k+1}-\bb x_{k-i}}\leq \delta_2$, then
\begin{equation}
\norm{\bb \psi^\tau(\bb x_{k+1},\bb x_{k+1},\dots,\bb x_{k+1}) - \bb \psi^\tau(\bb x_{k+1},\bb x_{k},\dots,\bb x_{k-\mu+1})} \leq \frac{\epsilon}{3}.
\label{eq:est3}
\end{equation}
By the UBD property of \eqref{eq:disc} with $\delta_3 := \min\{\delta_2, \frac{D_\epsilon}{2}\}$, there exists a $\tau_{5}>0$ depending on $\delta_3$ and $X_{0}^\epsilon$ so that $\bb x_{k+1}\in \displaystyle \bigcap_{0\leq i\leq \mu-1} \overline{B_{\delta_3}(\bb x_{k-i})}$ for $0<\tau<\tau_{5}$. In other words, $\max_{0\leq i\leq \mu-1}\norm{\bb x_{k+1}-\bb x_{k-i}}\leq\delta_3$ if $0<\tau<\tau_{5}$. Thus combining \eqref{discCons} with \eqref{eq:est1}-\eqref{eq:est3}, if $0<\tau<\tau':=\min \{\tau_2,\dots, \tau_5\}$,
\begin{align*}
\norm{\bb\psi(\bb x_{k+1})-\bb c} &\leq \norm{\bb \psi(\bb x_{k+1})-\bb \psi^\tau(\bb x_{k+1},\bb x_{k+1},\dots,\bb x_{k+1})}  \\
& \hskip 2mm + \norm{\bb \psi^\tau(\bb x_{k+1},\bb x_{k+1},\dots,\bb x_{k+1}) - \bb \psi^\tau(\bb x_{k+1},\bb x_{k},\dots,\bb x_{k-\mu+1})} \\
& \hskip 2mm + \underbrace{\norm{\bb \psi^\tau(\bb x_{k+1},\bb x_{k},\dots,\bb x_{k-\mu+1})-\bb c}}_{=\norm{\bb c^\tau-\bb c}} \leq \epsilon.
\end{align*}
In other words, we have shown that if $0<\tau< \tau'$,
\begin{equation}
\bb x_{k+1} \in \bb \psi^{-1}(\overline{B_{\epsilon}(\bb c)}) \cap \left(\bigcap_{0\leq i\leq \mu-1} \overline{B_{\frac{D_\epsilon}{2}}(\bb x_{k-i})}\right).
\label{eq:preImageNeighbor}
\end{equation}
We now claim that $\bb x_{k+1} \in X_{0}^\epsilon$, if $0<\tau< \tau'$. Suppose not, then $\bb x_{k+1} \in \bb \psi^{-1}(\overline{B_{\epsilon}(\bb c)})-X_{0}^\epsilon$. In particular, $\bb x_{k+1} \in X_{j}^\epsilon$ for some $j\neq 0$, which leads to a contradiction since by \eqref{positiveDist} and \eqref{eq:preImageNeighbor}, for any $0\leq i\leq \mu-1$,
\begin{equation*}
0<D_\epsilon \leq d(\bb x_{k+1}, X_{0}^\epsilon) \leq \underbrace{d(\bb x_{k-i}, X_{0}^\epsilon)}_{=0}+\underbrace{d(\bb x_{k-i}, \bb x_{k+1})}_{\leq D_\epsilon/2}.
\end{equation*}
Hence, the strong induction step is shown and Claim \ref{claim:2} is true for $\tau^*:=\min \{\tau_\mu, \tau'\}$.

Finally, the long term stability follows from Claim \ref{claim:1} and \ref{claim:2}. Since $X_{0}^\epsilon$ is bounded and for all $k\in \mathbb{N}$, $\bb x(t_k) \in X_0\subset X_{0}^\epsilon$ and $\bb x_k \in X_{0}^\epsilon$ if $0<\tau<\tau^*$, we can conclude for any fixed $0<\epsilon<\epsilon_1$,
\[
\norm{\bb x(t_k)-\bb x_{k}} \leq \sup_{\bb x, \bb y \in X_{0}^\epsilon} \norm{\bb x-\bb y} = diam(X_{0}^\epsilon) =: C.
\]
\end{proof}

\begin{remark} \label{avgIdenProof}
Note that if $\bb \psi^\tau$ does not depend on $\tau$ explicitly, then the main stability result can be shown readily by using the averaging identity of Corollary \ref{avgIdentity}. Specifically, we have by the averaging identity that,
\[
\norm{\bb\psi(\bb x_{k+1})-\bb c} &= \norm{\bb\psi^\tau (\bb x_{k+1},\dots,\bb x_{k+1})-\bb c} \\
&\leq \norm{\bb\psi^\tau (\bb x_{k+1},\dots,\bb x_{k+1})- \bb \psi^\tau(\bb x_{k+1},\bb x_{k},\dots,\bb x_{k-\mu+1})} \\
&\hskip 2mm + \norm{\bb \psi^\tau(\bb x_{k+1},\bb x_{k},\dots,\bb x_{k-\mu+1})-\bb c}.
\] Thus, by continuity of $\bb \psi^\tau$, Lemma \ref{initLemma} and the uniformly bounded displacement property, \ref{eq:preImageNeighbor} holds for sufficiently small $\tau$ and the proof proceeds similarly as in the equicontinuous case.
\end{remark}

\begin{remark}
We note that existence of a bounded connected component can be difficult to establish in general, as we discuss in the conclusion. For applications arising from physics, the energy function is a scalar conserved quantity $\psi(\bb x)$ and often satisfies the coercive property; $|\psi(\bb x)| \rightarrow \infty$ as $\norm{\bb x}\rightarrow 0$. In this case, it is well-known that nonempty preimage of $\bb \psi$ is bounded.
\end{remark}

\begin{remark}
Similarly, it may be difficult to establish in general whether a given bounded connected component of $\bb \psi^{-1}(\bb c)$ has a bounded locally finite neighborhood. Indeed, in the special case when $\bb \psi^{-1}(\bb c)$ has only finitely many connected components, it follows from definition that every bounded connected component has a bounded locally finite neighborhood.
\end{remark}

\subsection{Long-term stability in practice}
\label{sec:CMIP}
We conclude this section with a discussion on the effect of error accumulation for conservative methods. 

Recall that the main stability result of Theorem \ref{mainTheorem} holds provided we can ensure \ref{eq:preImageNeighbor} holds for some $\epsilon>0$ with a nonzero separation distance $D_\epsilon$ between $X_{0}^\epsilon$ and $\bigcup_{0\neq j \in J^\epsilon} X_j^\epsilon$. However, even for conservative methods, $\bb \psi^\tau\discdep{\bb x^\tau}{k+1}\neq \bb \psi^\tau\discdep{\bb x^\tau}{k}$ in practice due to round-off error of inexact arithmetic operations in computing $\bb \psi^\tau$. Moreover, $\bb x_{k+1}$ is often an inexact solution to some iterative method for an implicit method. These errors, while negligible at each time step, can accumulate over long term to be sufficiently large, leading to $\bb x_{k+1} \notin \bb \psi^{-1}(\overline{B_{\epsilon}(\bb c)})$ and violating in the hypotheses of the main theorem. We stress that the error accumulation discussed here is inherent for any finite precision machine when inexact computations are performed over many iterations. In contrast to non-conservative methods, conservative methods are solely limited by these error accumulations depending on machine precision and tolerances used within the method, which we characterize next.

Let $\epsilon_{mach}$ be a fixed machine precision and $\delta_{tol}$ be a fixed tolerance used as the stopping criterion within the iterative procedure of a given conservative method. Then on each successive time step $t_k$, we denote the error $E_k$ accumulated in computing $\bb \psi^\tau$ on a finite precision machine as
\[
\norm{\bb \psi^\tau\discdep{\bb x^\tau}{k} -\bb \psi^\tau\discdep{\bb x^\tau}{k-1}} \leq E_k(\epsilon_{mach},\delta_{tol}).
\]
Thus by triangle inequality, the error in $\bb \psi^\tau$ after $N$ time steps can be estimated as
\[
\norm{\bb \psi^\tau\discdep{\bb x^\tau}{N} -\bb c^\tau}\leq \sum_{k=\mu-1}^N \norm{\bb \psi^\tau\discdep{\bb x^\tau}{k} -\bb \psi^\tau\discdep{\bb x^\tau}{k-1}} \leq \sum_{k=\mu-1}^N E_k(\epsilon_{mach},\delta_{tol}),
\] where $\bb c^\tau := \bb \psi^\tau\discdep{\bb x^\tau}{\mu-1}$ from \eqref{discCons}. Moreover, suppose $E_k$ can be bounded uniformly by some constant $C_a(\epsilon_{mach},\delta_{tol})>0$, then
\[
\norm{\bb \psi^\tau\discdep{\bb x^\tau}{N} -\bb c^\tau} \leq C_aN.
\] In other words, the error in $\bb \psi^\tau$ grows linearly with $N$ in the worst case. However in practice, there may be cancellations within the expressions of the conservative method which can lead to sharper estimates of these round-off errors, as will be illustrated in the numerical example of Section \ref{sec:numeric}. This leads to the following definition of the class of conservative methods with \emph{error accumulation rate} $s$ for some $0<s\leq 1$. 

\begin{definition}
Fix a machine precision $\epsilon_{mach}$ and a tolerance $\delta_{tol}$. A conservative $\mu$-step method $\bb F^\tau$ is said to have an error accumulation rate $s$ if there exists some constants $0<s(\epsilon_{mach},\delta_{tol})\leq 1$ and $C_a(\epsilon_{mach},\delta_{tol})>0$ such that,
\[
\norm{\bb \psi^\tau\discdep{\bb x^\tau}{N} -\bb c^\tau} \leq C_a N^{s}.
\]
\end{definition} 
\begin{remark}
Note that the optimal error accumulation rate $s=\frac{1}{2}$ is known as Brouwer's law \cite{Bro37} and can be achieved for certain linear multi-step methods \cite{GraAE04, Qui94} and Runge-Kutta methods \cite{HaiMclRaz08}.
\end{remark}

Now we state the {\bf arbitrarily long-term stability} theorem on finite precision machines.

\begin{theorem} \label{mainTheoremIP}
Suppose the hypotheses of Theorem \ref{mainTheorem} (Main stability theorem) are satisfied and the $\mu$-step method $\bb F^\tau$ has an error accumulation rate $s$. If $X_0$ has a bounded LFN with $\bb x_0 \in X_0$, then there exists a positive integer $N_{max}$ depending only on $\epsilon_{mach}$ and $\delta_{tol}$, a positive constant $C$ independent of $N_{max}$ and a positive constant $\tau^*$ independent of $\tau$ and $k$ such that if $0<\tau<\tau^*$, then for all $0\leq k\leq N_{max}$,
\begin{equation*}
\norm{\bb x(t_k)-\bb x_{k}} \leq C.
\end{equation*}
\end{theorem}

\begin{proof}
We highlight the differences in the proof, as it is nearly identical to the proof of Theorem \ref{mainTheorem}. For brevity, we shall focus on the case when $\bb \psi^\tau$ does not explicitly depend on $\tau$, as the same conclusion follows for the equicontinuous case (with possibly a smaller $N_{max}$). Since the quantity $\norm{\bb \psi^\tau\discdep{\bb x^\tau}{k} -\bb c^\tau}\neq 0$ with finite precision arithmetic, by Remark \ref{avgIdenProof}, we need to instead establish the following estimate,
\begin{equation}
\norm{\bb \psi^\tau\discdep{\bb x^\tau}{k+1} -\bb c} \leq \norm{\bb \psi^\tau\discdep{\bb x^\tau}{k+1} -\bb c^\tau}+\norm{\bb c^\tau -\bb c} \leq \frac{\epsilon}{2}, \label{eq:modEst}
\end{equation} for $0<\epsilon<\epsilon_1$ where $\epsilon_1$ is the largest radius around $\bb c$ for which $X_{0}^\epsilon$ is separated from the other connected components $X_j^\epsilon$. 
By Lemma \ref{initLemma}, $\norm{\bb c^\tau -\bb c}\leq \frac{\epsilon}{4}$ for sufficiently small $\tau$. Moreover, since $\bb F^\tau$ is assumed to have an error accumulation rate $s$, then $\norm{\bb \psi^\tau\discdep{\bb x^\tau}{k+1} -\bb c^\tau} \leq C_a N_{max}^{s}$ for all $k+1\leq N_{max}$. Thus, the estimate \eqref{eq:modEst} follows if $N_{max}:=\left(\frac{\epsilon}{4 C_a}\right)^{1/s}$. Finally, we also note that $C:=diam(X^\epsilon_{0})$ as in the main theorem and is independent of $N_{max}$.
\end{proof}

\section{Numerical example: Elliptic curve}
\label{sec:numeric}

We now illustrate the long-term stability theorem for an autonomous system with a specific conserved quantity in the form of an elliptic curve. To obtain numerical results, we use conservative methods derived from the multiplier method \cite{WanBihNav17a}, which was developed as a systematic approach to construct conservative discretizations for general dynamical systems. The multiplier method differs from other general conservative methods, such as projection methods or discrete gradient methods, in that it does not involve projection nor require expressing the right hand side of the ODE system as a skew-symmetric tensor applied to gradients of the first integrals. Also, the multiplier method can readily be applied to dynamical systems without transformation and to conserved quantities which depend explicitly on time \cite{WanBihNav17a}. Furthermore, nonlinear multistep multiplier methods have also been developed in \cite{Wan18} to achieve higher order accuracy.

In the following, we have applied the multiplier method detailed in \cite{WanBihNav17a} and \cite{Wan18} to obtain 1-step, 2-step and 3-step conservative methods for the following autonomous system

\begin{equation}
\bb F\cont{\bb x} :=
\begin{pmatrix}
\dot{x}\\
\dot{y}
\end{pmatrix}- 
\begin{pmatrix}2y \\
 3 x^2+a
\end{pmatrix}=\bb 0, \hskip 10pt \bb x(0) = \bb x_0, \label{ECsystem}
\end{equation} where $a\in\mathbb{R}$. Multiplying $\bb F$ by so-called multiplier matrix $\Lambda(\bb x) =\begin{pmatrix}-3x^2-a &2y\end{pmatrix}$ shows that \begin{equation*}
\psi(\bb x):= y^2-x^3-ax
\end{equation*} is a conserved quantity of \eqref{ECsystem}. Indeed, if $\bb x$ is the unique solution to \eqref{ECsystem},
\[
0&=\Lambda(\bb x) \bb F\cont{\bb x} = (-3x^2-a)(\dot{x}-2y)+2y(\dot{y}-3x^2-a) = D_t \psi(\bb x)
\]

Applying the multiplier method of \cite{WanBihNav17a} to \eqref{ECsystem}, we obtained the following 1-step conservative discretization
\begin{equation}
\begin{pmatrix}
\dfrac{x_{k+1}-x_{k}}{\tau}\\
\dfrac{y_{k+1}-y_{k}}{\tau}
\end{pmatrix}=
\begin{pmatrix}\dfrac{}{}y_{k+1}+y_{k} \\
\dfrac{}{}x_{k+1}^2+x_{k+1}x_k+x_k^2+a
\end{pmatrix}, \label{ECDisc1}
\end{equation} which conserves exactly $\psi^\tau(\bb x_k):=\psi(\bb x_k)$.
Using the higher order versions of multiplier methods of \cite{Wan18}, we obtained the following 2-step conservative discretization
\begin{equation}
\begin{pmatrix}
\dfrac{x_{k+1}-x_{k-1}}{2\tau}\\
\dfrac{y_{k+1}-y_{k-1}}{2\tau}
\end{pmatrix}=
\begin{pmatrix}\dfrac{}{}y_{k+1}+y_{k-1} \\
\dfrac{}{}x_{k+1}^2+x_{k+1}x_{k-1}+x_{k-1}^2+a
\end{pmatrix}, \label{ECDisc2}
\end{equation} which conserves exactly $\psi^\tau(\bb x_k,\bb x_{k-1}):=\frac{1}{2}\left(\psi(\bb x_k)+\psi(\bb x_{k-1})\right)$. Moreover, we also have the 3-step conservative discretization
\begin{equation}
\begin{pmatrix}
\dfrac{11x_{k+1}-18 x_{k}+9x_{k-1}-2x_{k-2}}{6\tau}\\
\dfrac{11y_{k+1}-18 y_{k}+9y_{k-1}-2y_{k-2}}{6\tau}
\end{pmatrix}=
\begin{pmatrix} \dfrac{11y_{k+1}^2-18 y_{k}^2+9y_{k-1}^2-2y_{k-2}^2}{11y_{k+1}-18 y_{k}+9y_{k-1}-2y_{k-2}} \\
\dfrac{11x_{k+1}^3-18 x_{k}^3+9x_{k-1}^3-2x_{k-2}^3}{11x_{k+1}-18 x_{k}+9x_{k-1}-2x_{k-2}}+a
\end{pmatrix}, \label{ECDisc3}
\end{equation} which conserves exactly $\psi^\tau(\bb x_k,\bb x_{k-1}, \bb x_{k-2}):=\frac{1}{6}\left(11\psi(\bb x_k)-7\psi(\bb x_{k-1})+2\psi(\bb x_{k-2})\right)$. 

In Appendix \ref{sec:convOrder}, we observed numerically that the three conservative methods \eqref{ECDisc1}, \eqref{ECDisc2}, \eqref{ECDisc3} are of order two, two and three, respectively. 

\begin{remark} \eqref{ECDisc1} can also be derived using the average vector field method \cite{QuiMcL08}, since closed form integration can be performed for polynomials.
\end{remark}

\begin{remark} Note that the discretized conserved quantities $\psi^\tau$ of \eqref{ECDisc2} and \eqref{ECDisc3} satisfies the averaging identity of Corollary \ref{avgIdentity}.
\end{remark}

\begin{remark} For the nonlinear 1-step, 2-step and 3-step methods of \eqref{ECDisc1}, \eqref{ECDisc2} and \eqref{ECDisc3}, a fixed point type argument can be used to show existence and uniqueness of $\bb x_{k+1}$. We showed the UBD property for \eqref{ECDisc1} in the Appendix and leave details of showing the UBD property for nonlinear multistep conservative methods in \cite{Wan18}. 

\end{remark}

For any $a,b \in\mathbb{R}^n$, it is well-known that the elliptic curve $\psi(\bb x)=b$ has at most two connected components. In particular, if the sign of the discriminant of the cubic polynomial $p(\bb x):=\psi(\bb x)-b$ is given by $\Delta(p):=4a^3+27b^2$ is negative, the elliptic curve has two connected components with one bounded and the other unbounded. Otherwise, the elliptic curve has only one unbounded connected component if $\Delta(p)> 0$.

\subsection{Long-term stability of 1-step conservative method}

We first compare numerical results of the 1-step conservative method \eqref{ECDisc1} with standard first order explicit/implicit methods (Euler/Backward Euler) and symplectic second order explicit/implicit methods (St\"{o}rmer-Verlet/Midpoint) of \cite{HaiLubWan06}.
We considered the case of two connected components with $a=-1$ and $b=0.3849$ ($\Delta(p)<0$), where the bounded connected component of $X_0$ and unbounded connected component of $X_1$ are close to each other but separated as shown in Figure \ref{figs:ECexact} and \ref{figs:ECexactZoom}. In all five methods, the initial conditions were set to be $x_0 = 0.571$ and $y_0 = \sqrt{x_0^3+a x_0+b} \approx 8.33\times 10^{-3}$, which implies, for $\tau$ sufficiently small, the exact solution should remain within the bounded connected component of $X_0$. We have used an uniform time step size of $\tau = 0.3$ with $N=5\times 10^3$ time steps and we employed an absolute tolerance of $\delta_{tol}=5\times 10^{-16}$ with a maximum of $50$ Newton's iterations per time step for the implicit methods.

\begin{figure}[tbhp]
\centering
	\subfloat[Elliptic curve]{\label{figs:ECexact}\includegraphics[width=0.48\textwidth]{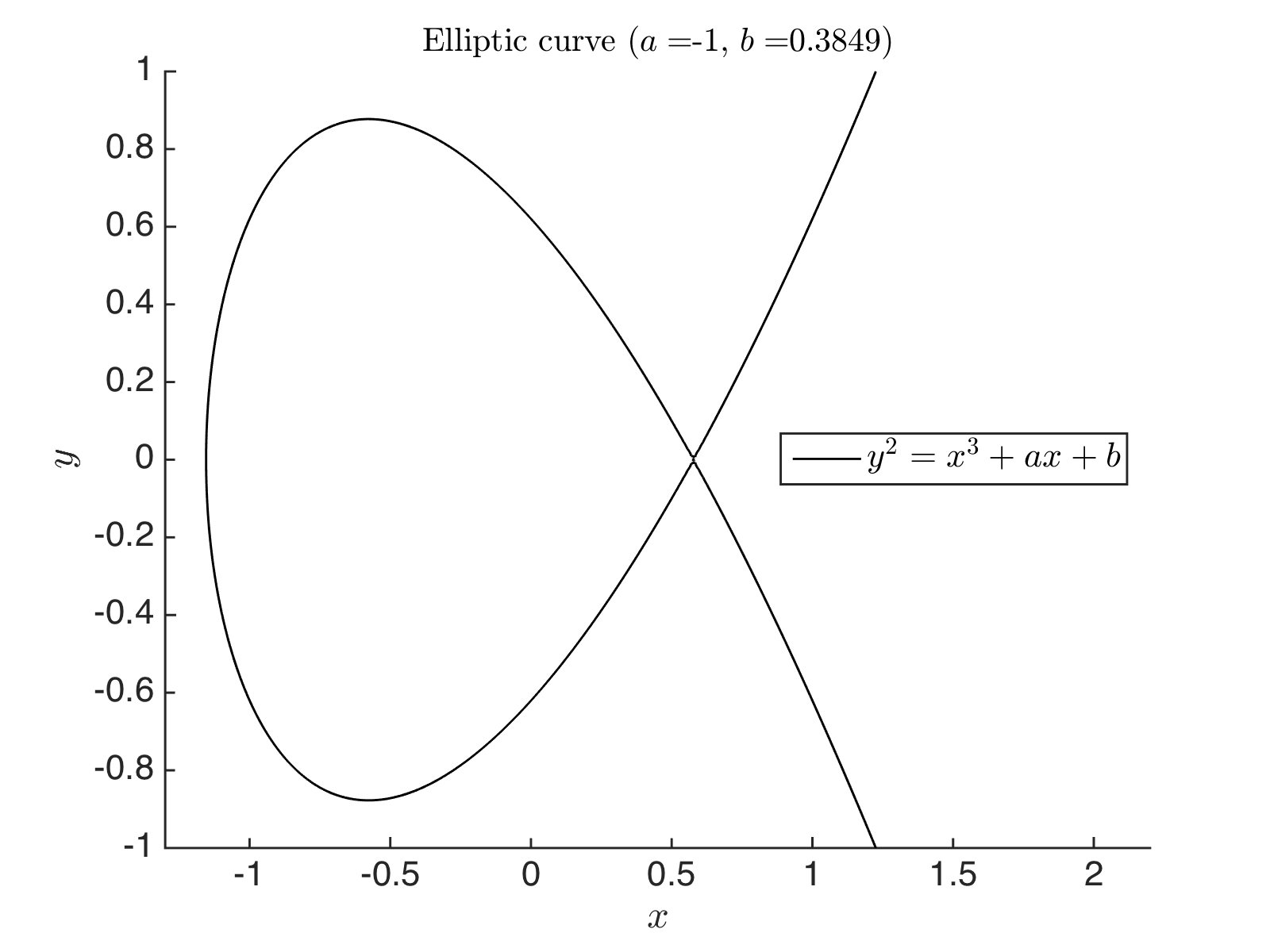}}
	\subfloat[Close-up of the gap between $X_0$ and $X_1$]{\label{figs:ECexactZoom}\includegraphics[width=0.48\textwidth]{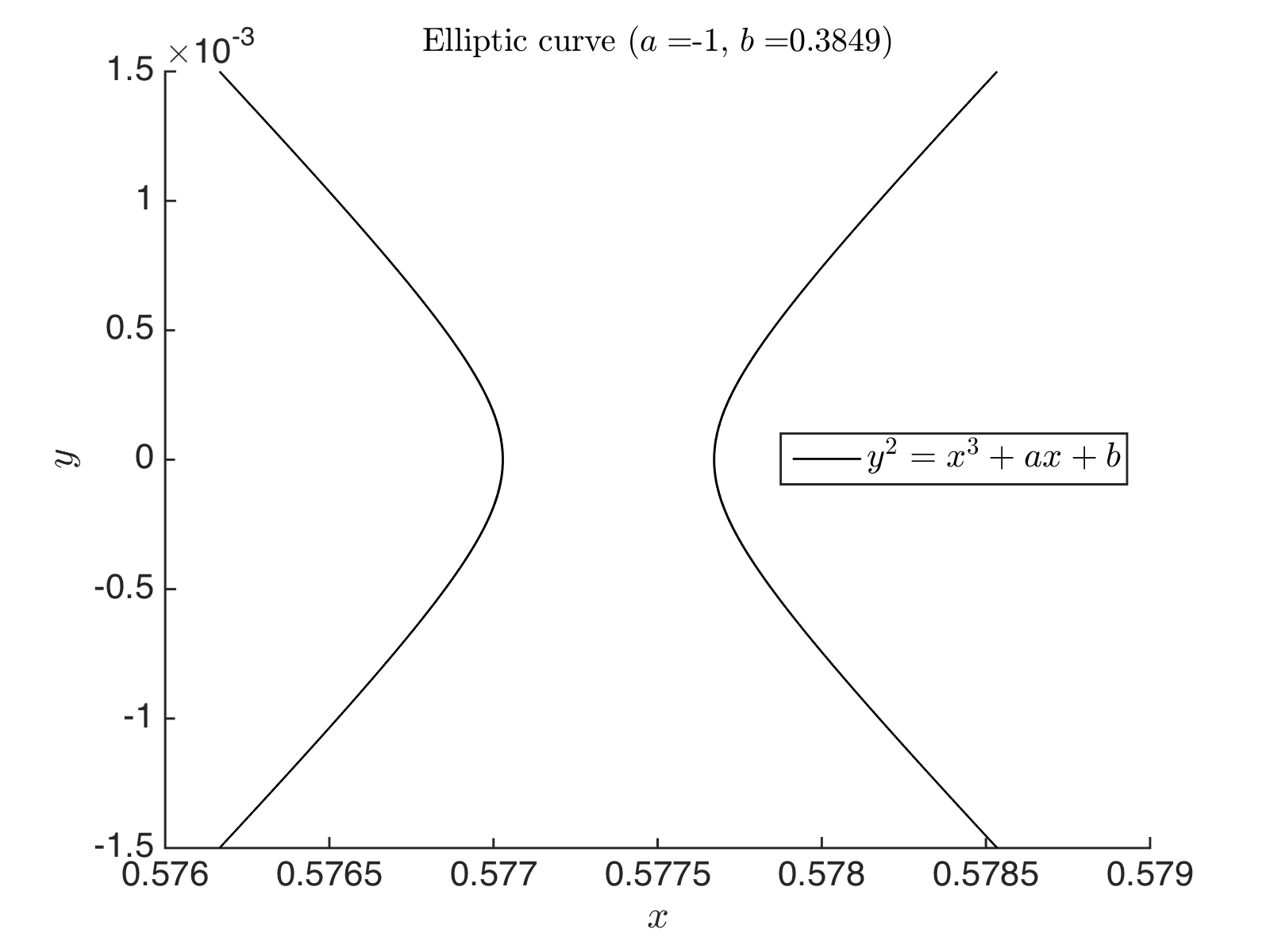}}
	\caption{Two connected components of the preimage of $\psi^{-1}(\{b\})$.}
\end{figure}

\begin{figure}[tbhp]
	\centering
	\label{figs:ECcompare}\includegraphics[width=0.65\textwidth]{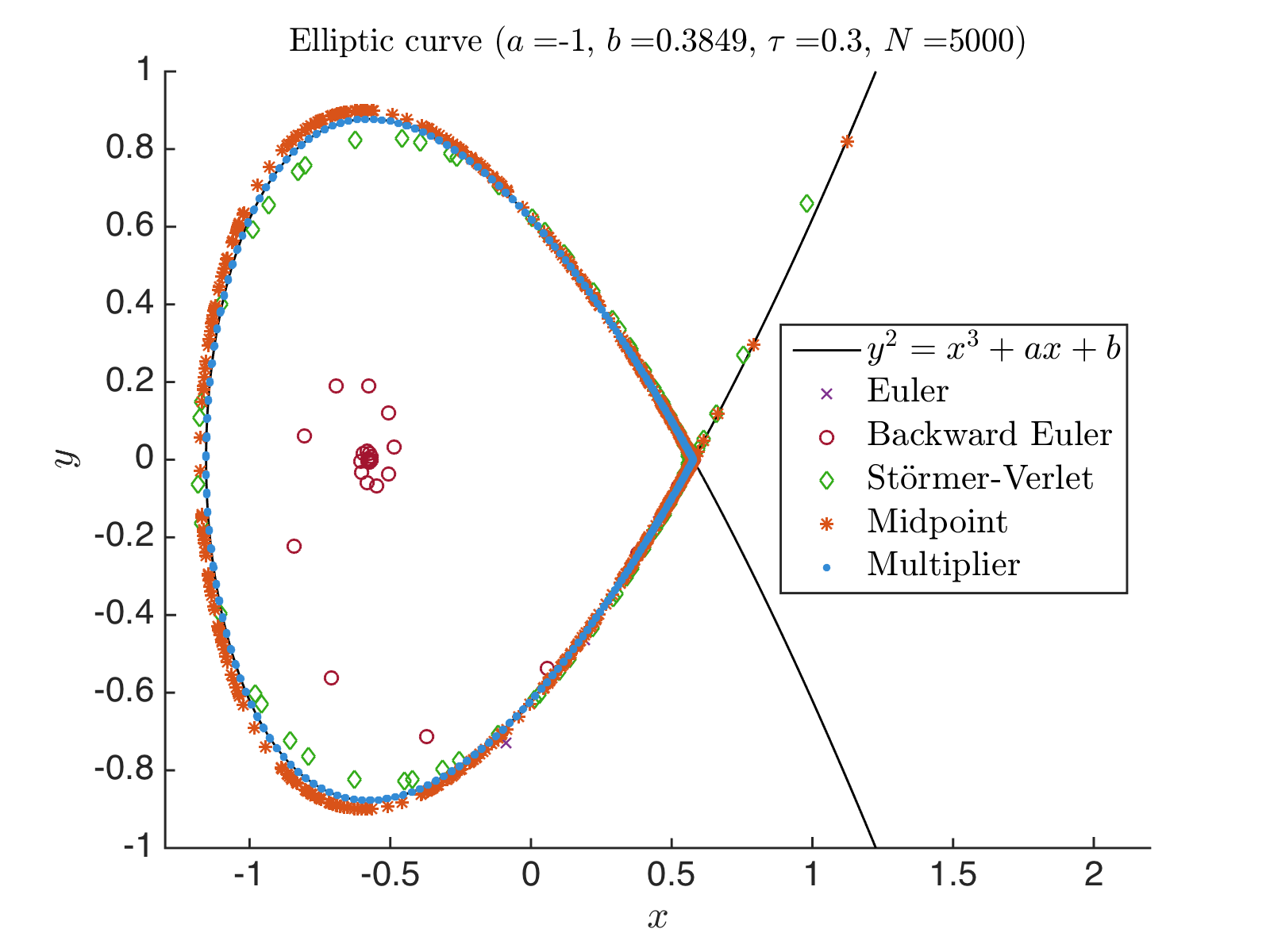}
	\vskip -5mm
	\caption{Comparison of first order standard methods (Euler/Backward Euler), second order symplectic methods (St\"ormer-Verlet/Midpoint) and the 1-step conservative method (Multiplier method).}
\end{figure}
Figure \ref{figs:ECcompare} shows that Euler's method gives an unbounded solution and Backward Euler method leads to a decaying solution to a fixed point $\bb x^* = \begin{pmatrix}-1/\sqrt{3},0\end{pmatrix}^T$. Figures \ref{figs:ECcompare} and \ref{figs:ECzoom1} show St\"ormer-Verlet method gives a solution which loops around the bounded connected component of $X_0$ once before exiting to the unbounded connected component of $X_1$. Similarly, Figures \ref{figs:ECcompare} and \ref{figs:ECzoom2} show the solution of the Midpoint method loops around $X_0$ longer than the St\"ormer-Verlet method before eventually exiting to $X_1$. In contrast, all three figures show that the 1-step multiplier method gives a solution which remains essentially on the bounded connect component of $X_0$ and indeed we observed an error in $\psi$ of $\displaystyle \max_{1\leq i \leq 5\times 10^3}|\psi(\bb x_i) - b|\sim 6.6\times 10^{-15}$.

\begin{figure}[tbhp]
\centering
	\subfloat[]{\label{figs:ECzoom1}\includegraphics[width=0.48\textwidth]{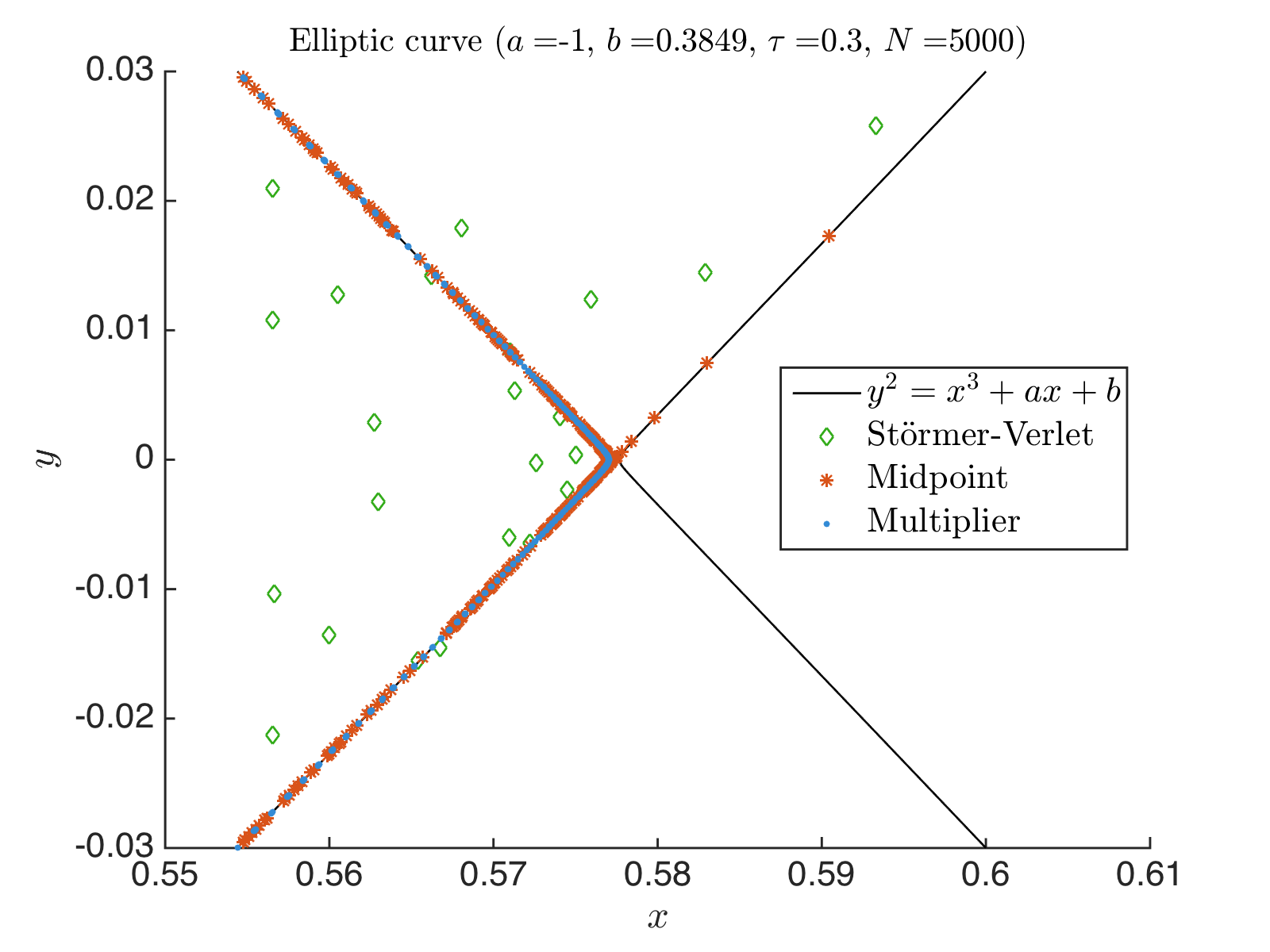}}
	\subfloat[]{\label{figs:ECzoom2}\includegraphics[width=0.48\textwidth]{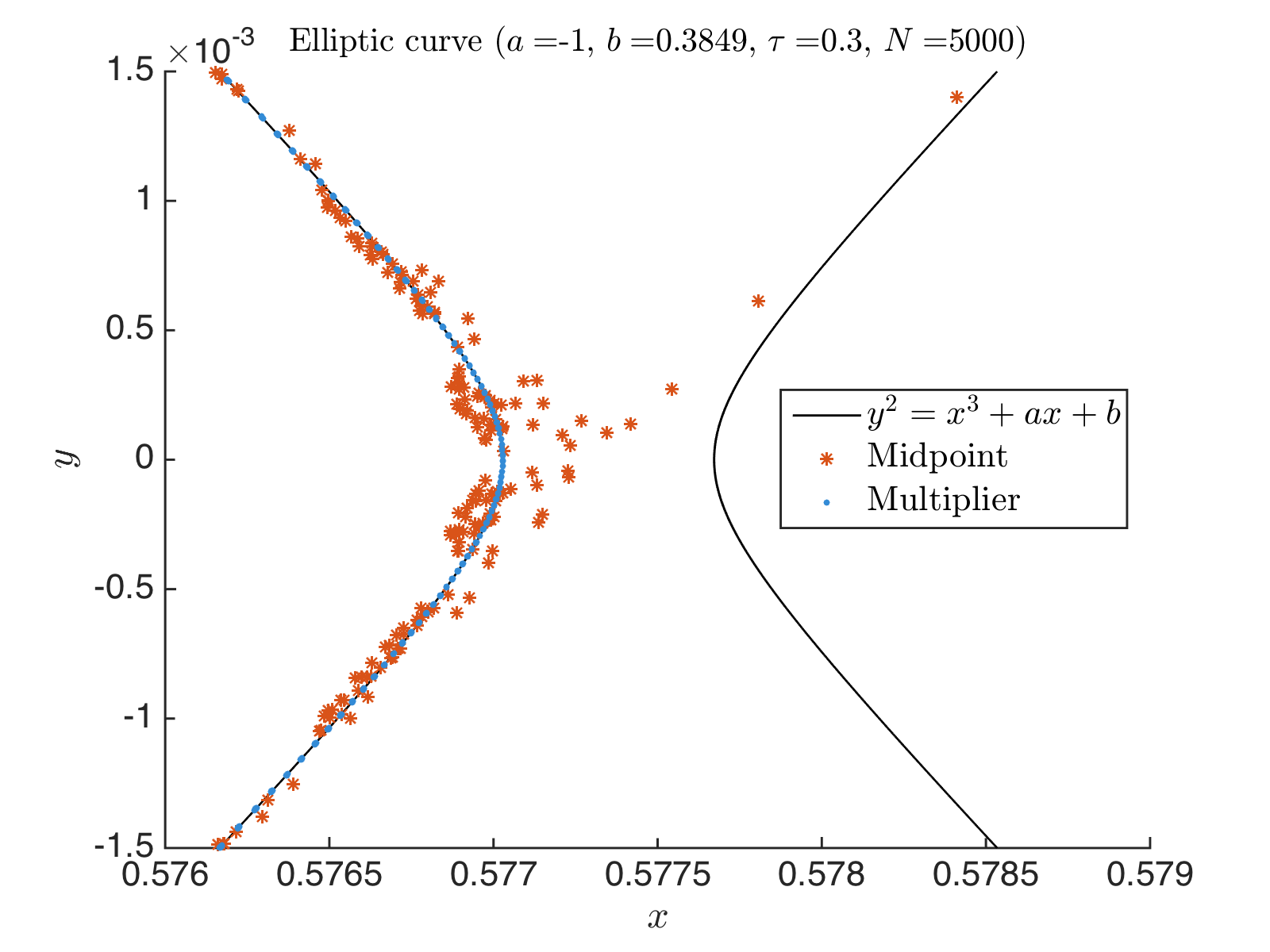}}
	\caption{Comparison between the St\"{o}rmer-Verlet, Midpoint and Multiplier method in Figure \ref{figs:ECzoom1} and close-up of the Midpoint and Multiplier method in Figure \ref{figs:ECzoom2}.}
\end{figure}
Next, we increase the number of time steps to $N=5\times 10^7$ while fixing all other parameters. As shown
in Figures \ref{figs:ECzoom1} and \ref{figs:ECzoom2}, the 1-step multiplier method again gives a solution which stays near the bounded connected component of $X_0$. As we increase the number of time steps, we expect an increase of the error in $\psi$ due to round-off error accumulation and inexact iterative solutions as discussed in Section \ref{sec:CMIP}. Indeed, we observed the error in $\psi$ now to be $\displaystyle \max_{1\leq i \leq 5\times 10^7}|\psi(\bb x_i) - b|\approx 1.1\times 10^{-12}$.

\begin{figure}[tbhp]
\centering
	\subfloat[]{\label{figs:ECMult}\includegraphics[width=0.48\textwidth]{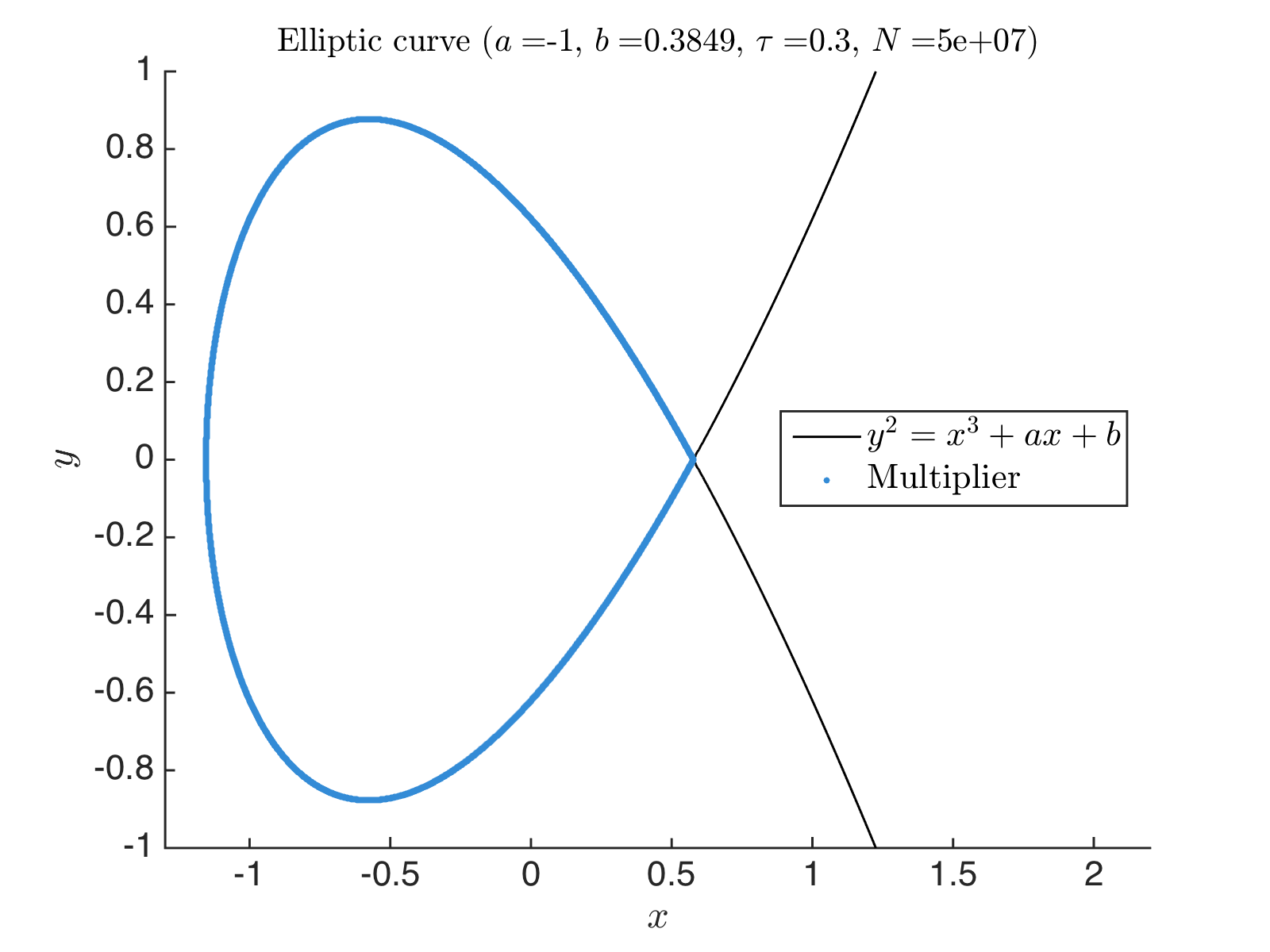}}
	\subfloat[]{\label{figs:ECMultZoom}\includegraphics[width=0.48\textwidth]{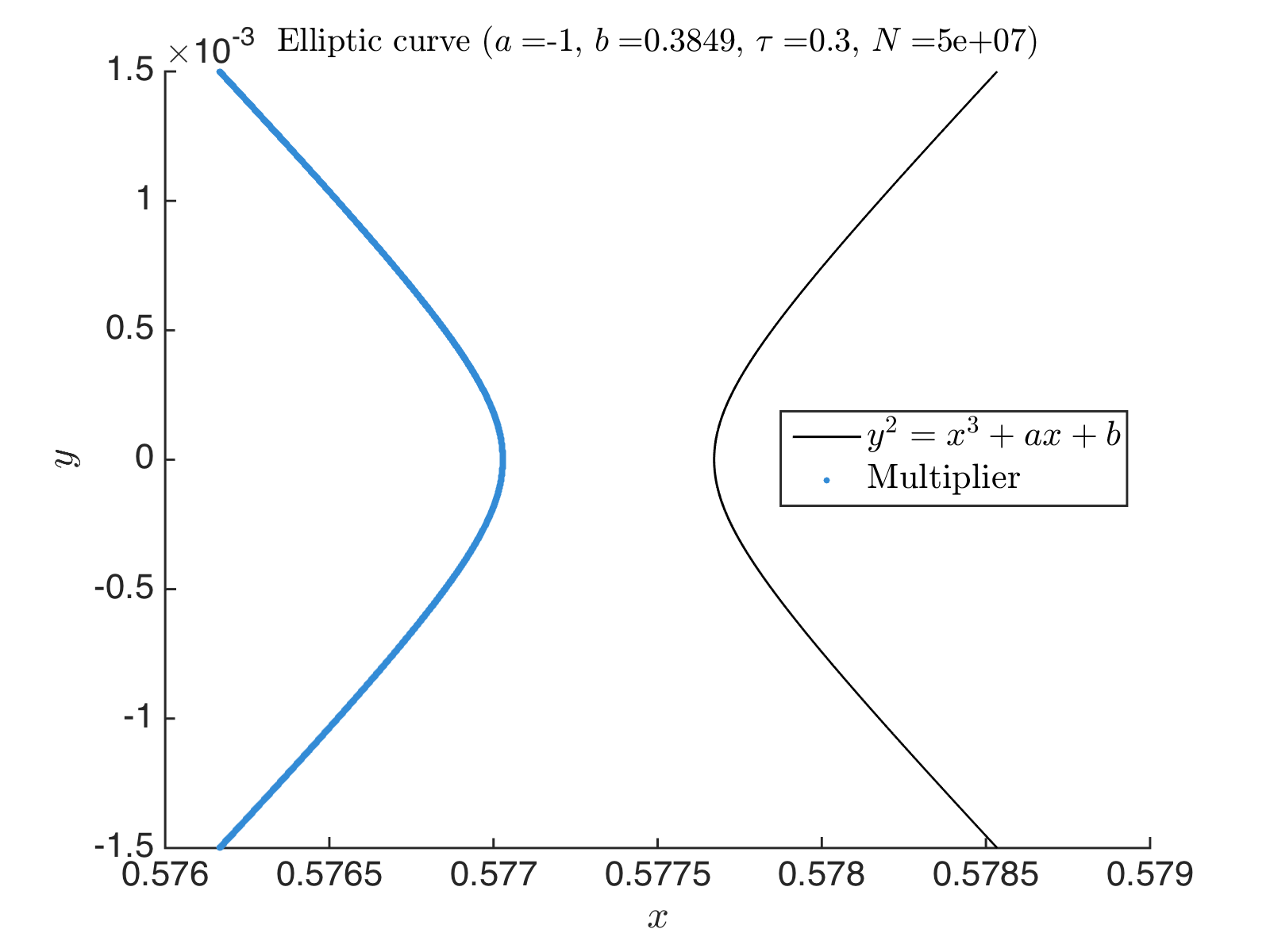}}
	\caption{Phase portraits of the 1-step conservative method for $N=5\times 10^7$.}
\end{figure}
\begin{figure}[tbhp]
\centering
	\label{figs:aEvsN}\includegraphics[width=0.65\textwidth]{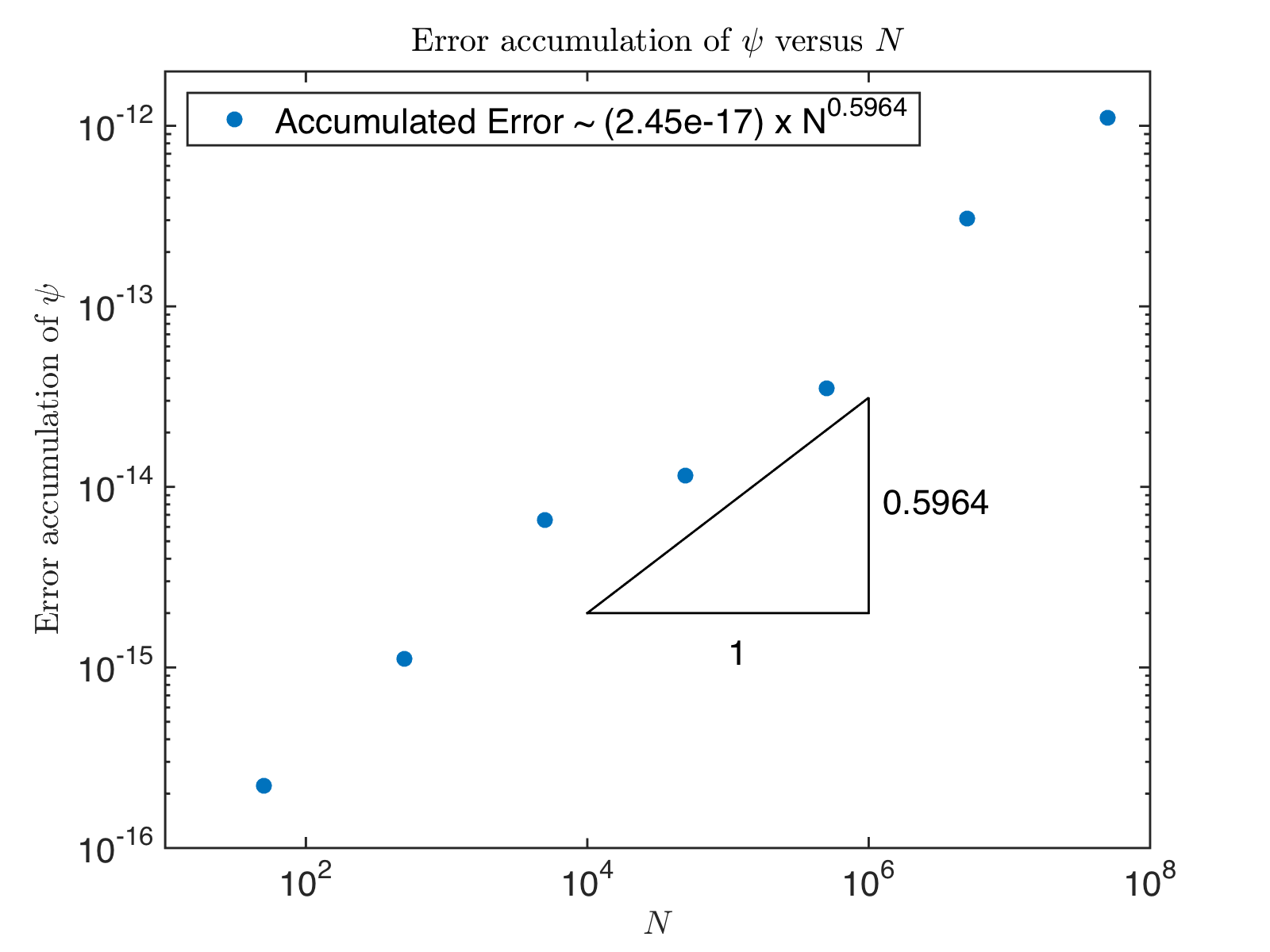}
	\caption{Error accumulation of $\psi$ versus $N$.}
\end{figure}

To investigate further on error accumulation of $\psi$ as $N$ increases, Figure \ref{figs:aEvsN} shows a log-log plot of the accumulated error $\displaystyle \max_{1\leq i \leq N}|\psi(\bb x_i) - b|$ for various $N$. By a linear regression, the accumulated error $E(N)$ was estimated to be $E\approx (2.45\times 10^{-17})\times N^{0.5964}$, where the error accumulation rate of $0.5964$ is due to inherent round-off cancellations within the conservative discretization. 

To estimate the maximum number of time steps $N_{max}$ as stipulated in Theorem \ref{mainTheoremIP}, we need the largest $\epsilon>0$ such that $\psi^{-1}((b-\epsilon, b+\epsilon))$ still has two connected components. In the present case of elliptic curve, we know that the two connected components coalesce into one single component precisely when the discriminant $\Delta(p)$ changes sign. Thus, computing $\Delta(p)=0$ gives $\epsilon \approx 1.8\times 10^{-7}$, which implies a maximum number of time steps $N_{max}=\left(\frac{\epsilon}{4 C_a}\right)^{1/s} \approx 6.9\times 10^{16}$ before the global error can grow in an unbounded fashion. This is in stark contrast to the previous four methods in which their solutions either decay to a fixed point or grow in an unbounded fashion.

\begin{remark}We elected here not to make comparison with projection-based conservative methods; methods which first evolve in time using traditional methods and after some time period project the discrete solution back onto the constraint of conserved quantities. While this approach can make any traditional method conservative, we note the long-term stability result may not hold for these methods if the composition of evolution and projection does not satisfy the UBD property.
\end{remark}

\subsection{Long-term stability of 2-step and 3-step conservative methods}

Next, we compare numerical results of the 2-step and 3-step conservative discretizations of \eqref{ECDisc2} and \eqref{ECDisc3}. 

In the following tests, we used the same initial conditions ($x_0 = 0.571$ and $y_0 \approx 8.33\times 10^{-3}$) as the 1-step conservative method but with an uniform time step size of $\tau = 0.003$ and a total of $N=5\times 10^5$ time steps\footnote{In this case, we observed that the larger step size of $\tau = 0.3$ was not sufficiently small for the long term stability to hold for the 3-step conservative method. For intermediate values of $\tau$, we observed the fixed point iteration may not converge or converge to a different solution (i.e. onto a different connected component) depending on the initial guess or initial condition.}. Furthermore, we employed a standard fixed point iteration to solve the implicit conservative methods with an absolute tolerance of $\delta_{tol}=5\times 10^{-16}$ and a maximum of 100 iterations per time step. The standard 4-th order Runge-Kutta method was used to initialize a guess for the fixed point iteration at each time step.

For a given $\mu$-step conservative method, the errors in the exact and the approximate conserved quantity $\psi, \psi^\tau$ are defined as
\[
\text{Error}[\psi] &:= \max_{k=\mu, \dots, N} |\psi(\bb x_k)-\psi(\bb x_0)|, \\
\text{Error}[\psi^\tau ] &:= \max_{k=\mu, \dots, N} |\psi^\tau(\bb x_k,\dots,\bb x_{k-\mu+1})-\psi(\bb x_0)|.
\]
From Table \ref{tab:compBootStrap}, we see that there are negligible differences for the errors in $\psi$ and $\psi^\tau$. This is to be expected as $\psi^\tau$ is a consistent approximation of $\psi$ for small $\tau$. Thus, for the remaining of this section, we will only list the errors in $\psi^\tau$.

\begin{table}[H]
\label{tab:compBootStrap}
\noindent \begin{centering}
\begin{tabular}{|c|c|c|}
\hline 
Bootstrap routine & 1-step method of \eqref{ECDisc1} & 4-th order Runge-Kutta method
\tabularnewline
\hline 
\hline 
Error$[\psi]$ of \eqref{ECDisc2}& $1.89\times 10^{-14}$& $4.29\times 10^{-14}$\tabularnewline
\hline
Error$[\psi^\tau]$ of \eqref{ECDisc2}&  $1.89\times 10^{-14}$& $4.29\times 10^{-14}$\tabularnewline
\hline 
Error$[\psi]$ of \eqref{ECDisc3}& $3.229\times 10^{-13}$& $1.186\times 10^{-13}$\tabularnewline
\hline
Error$[\psi^\tau]$ of \eqref{ECDisc3}& $3.226\times 10^{-13}$& $1.186\times 10^{-13}$\tabularnewline
\hline 
\end{tabular}
\par\end{centering}
\caption{Comparison of errors in $\psi$ and $\psi^\tau$ for the 2-step conservative method \eqref{ECDisc2} and 3-step conservative method \eqref{ECDisc3} with different bootstrapping routines.}\vskip -5mm
\end{table}

Next we verify the long term stability for the nonlinear multistep methods and bootstrapping routines, i.e. initializing the first $\mu$ values. Figures \ref{figs:2stepCMRKaE}-\ref{figs:3stepRKRKaE} show the log-log plot of the error accumulation of $\psi^\tau$ up to $N=5\times 10^7$ time steps for the 2-step and 3-step conservative method using different bootstrapping routines. Figures \ref{figs:ECMultiStep} and \ref{figs:ECMultiStepZoom} show the phase portraits for the 3-step conservative method. We omit the phase portrait for the 2-step method as it is visually indistinguishable from the 3-step method.

\begin{figure}[tbhp]
\centering
	\subfloat[2-step conservative method \eqref{ECDisc2} bootstrapped by 1-step conservative method \eqref{ECDisc1} ]{\label{figs:2stepCMRKaE}\includegraphics[width=0.48\textwidth]{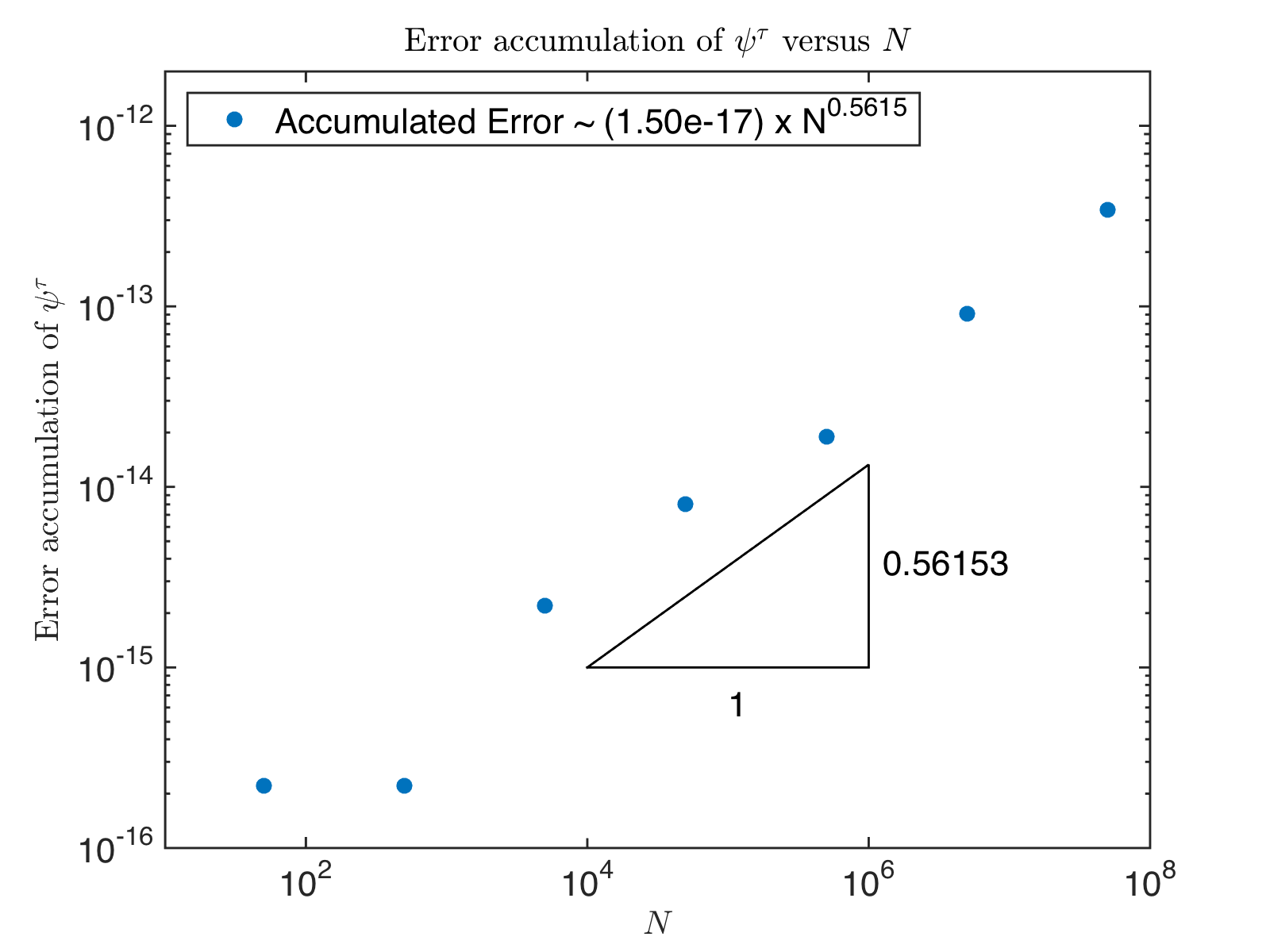}} \quad
	\subfloat[2-step conservative method \eqref{ECDisc2} bootstrapped by RK4 method]{\label{figs:2stepRKRKaE}\includegraphics[width=0.48\textwidth]{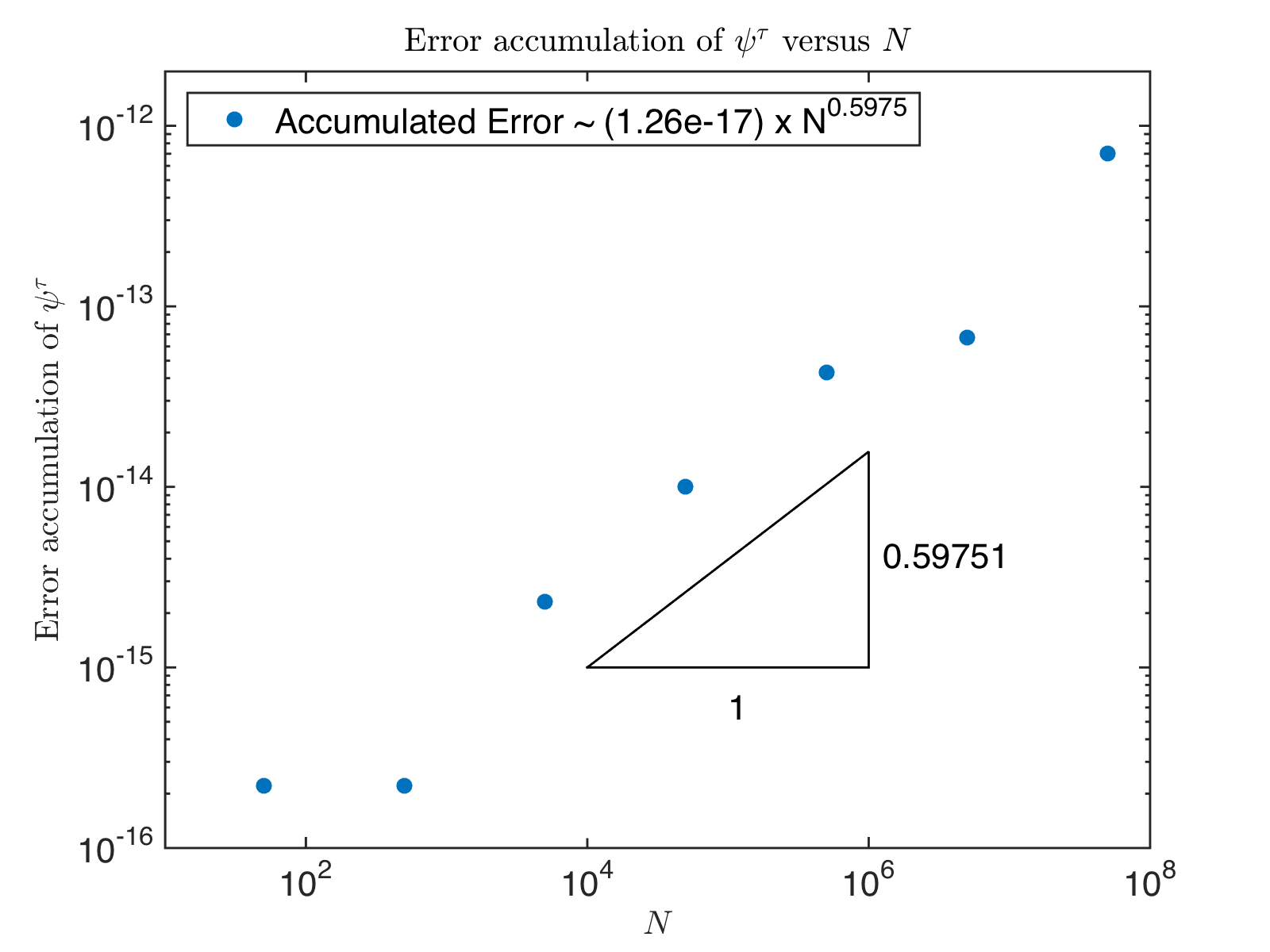}}\\
	\subfloat[3-step conservative method \eqref{ECDisc3} bootstrapped by 1-step conservative method  \eqref{ECDisc1}]{\label{figs:3stepCMRKaE}\includegraphics[width=0.48\textwidth]{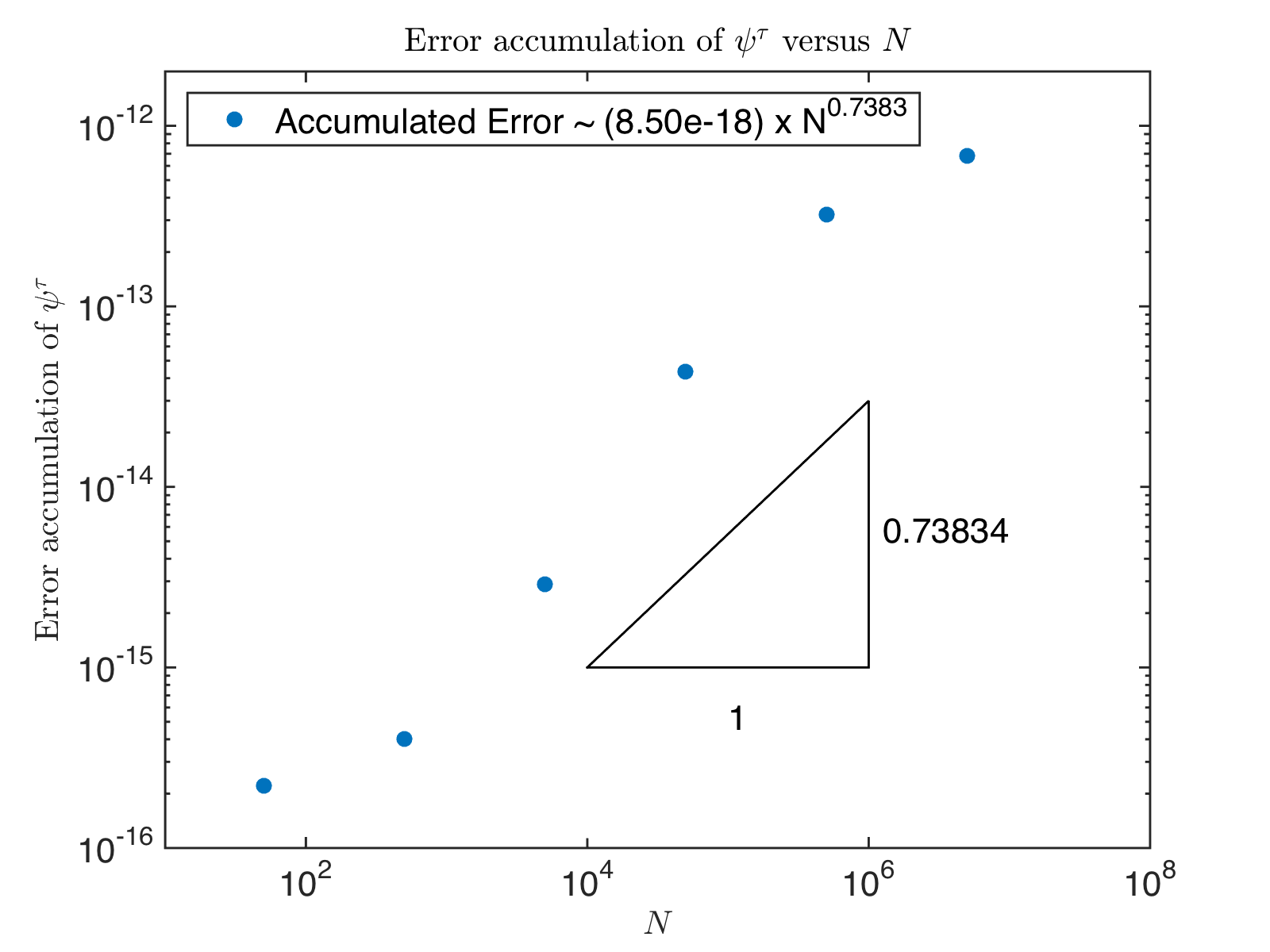}} \quad
	\subfloat[3-step conservative method \eqref{ECDisc3} bootstrapped by RK4 method]{\label{figs:3stepRKRKaE}\includegraphics[width=0.48\textwidth]{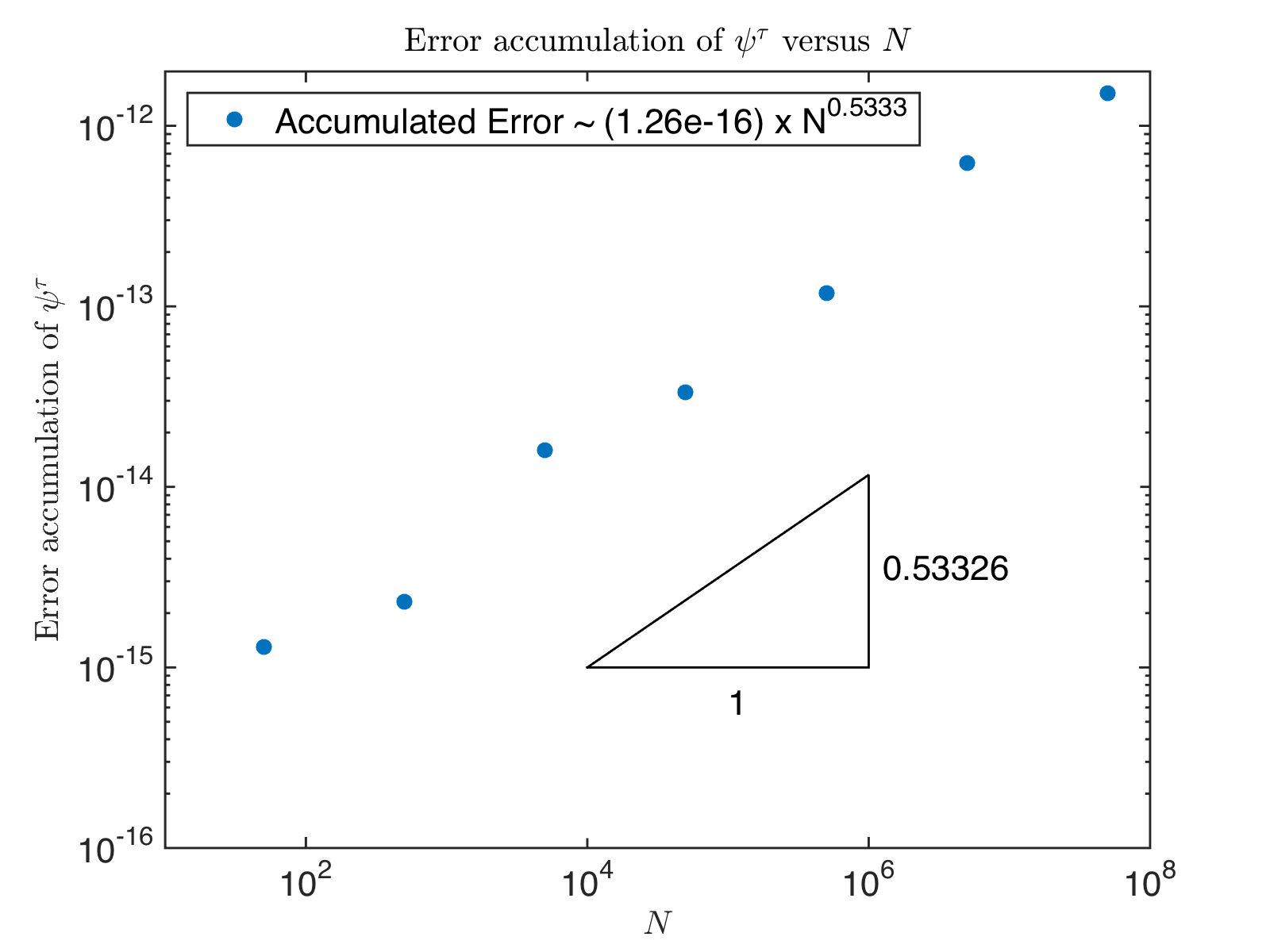}}
	\caption{Comparison error accumulation of the multistep conservative methods.}
\end{figure}
\begin{figure}[tbhp]
\centering
	\subfloat[]{\label{figs:ECMultiStep}\includegraphics[width=0.48\textwidth]{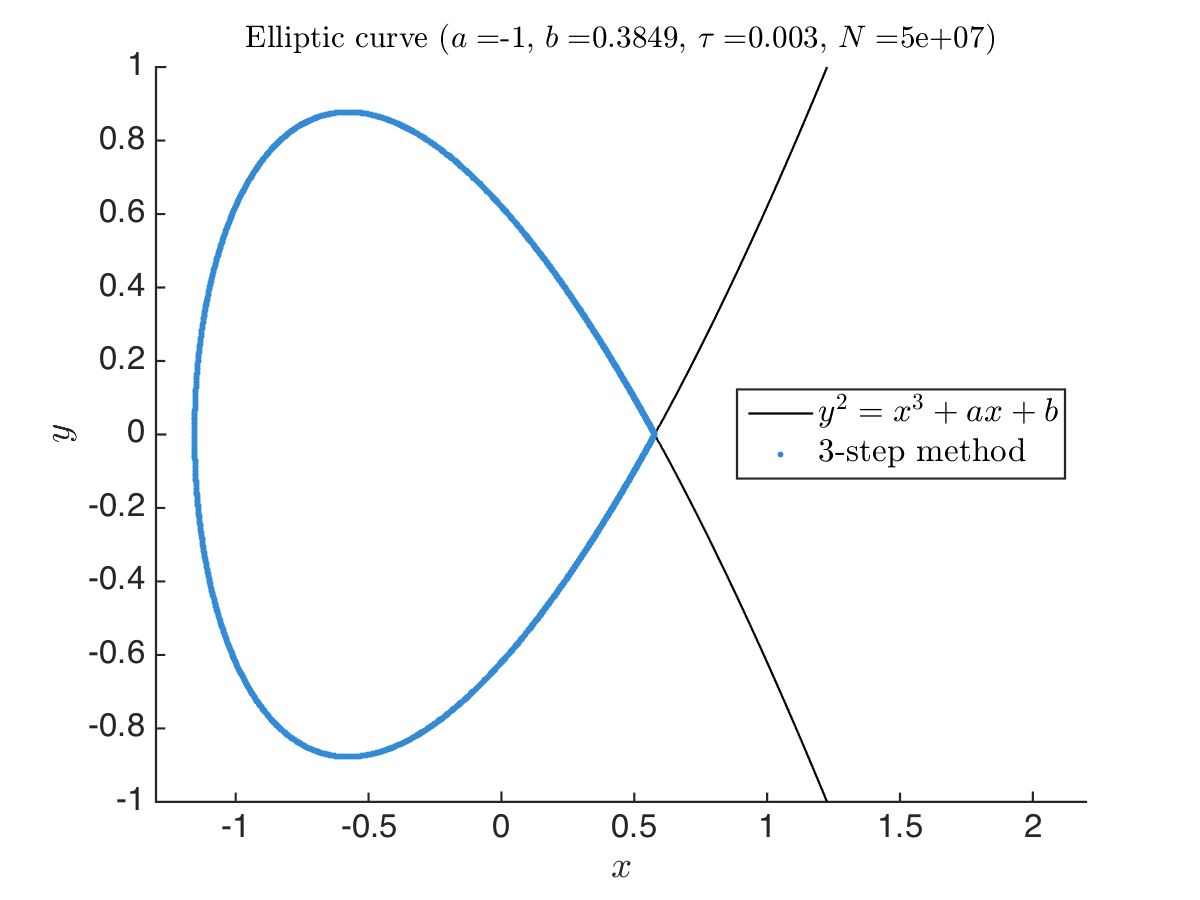}}
	\subfloat[]{\label{figs:ECMultiStepZoom}\includegraphics[width=0.48\textwidth]{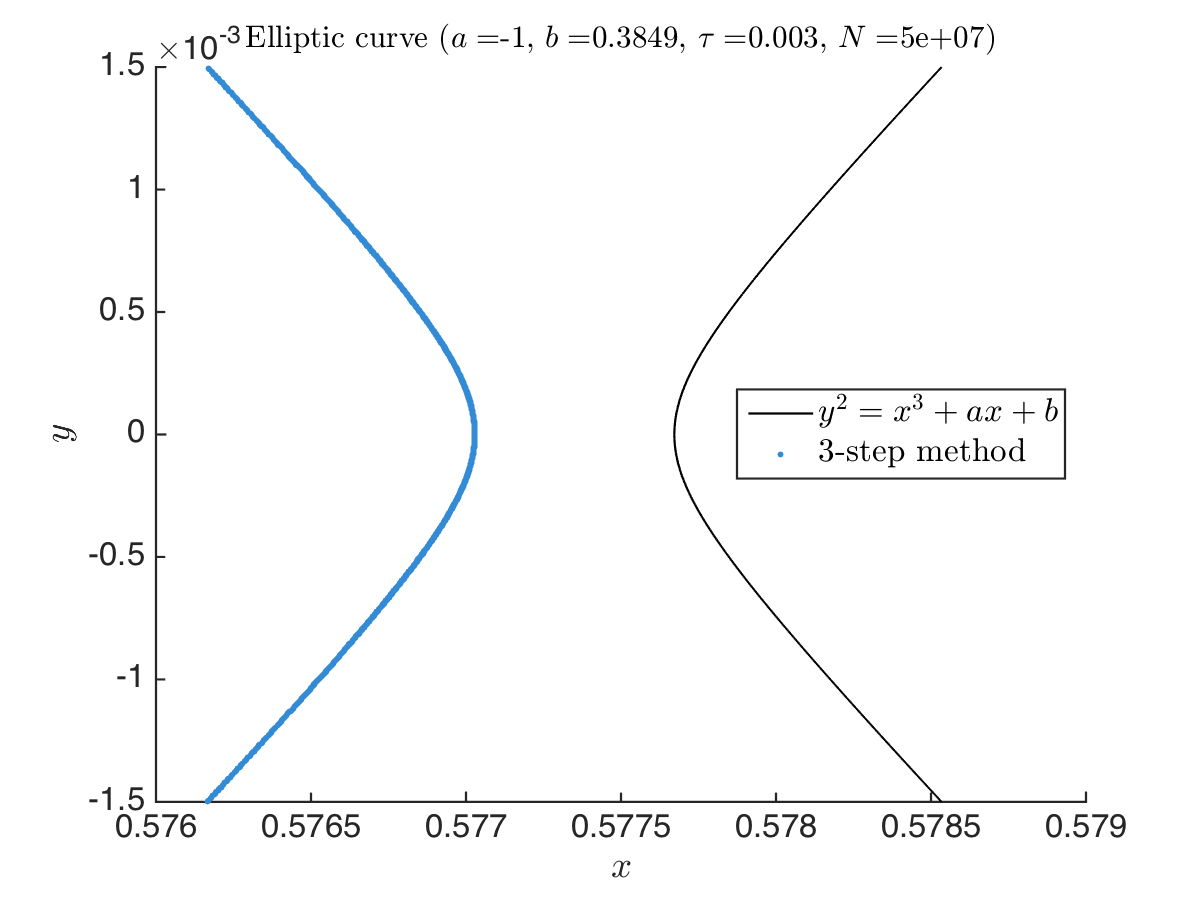}}
	\caption{Phase portraits of the 3-step conservative method for $N=5\times 10^7$.}
\end{figure}

Figures \ref{figs:eAccCM_RK} and \ref{figs:eAccRK_RK} show the differences in the error accumulation of $\psi^\tau$ on a linear scale for the 3-step conservative method. In particular, this demonstrates that bootstrapping a multistep conservative method with a higher order nonconservative method can still be beneficial in preserving conserved quantities over long term. 

\begin{figure}[tbhp]
\centering
	\subfloat[3-step conservative method \eqref{ECDisc3} bootstrapped by 1-step conservative method \eqref{ECDisc1}]{\label{figs:eAccCM_RK}\includegraphics[width=0.48\textwidth]{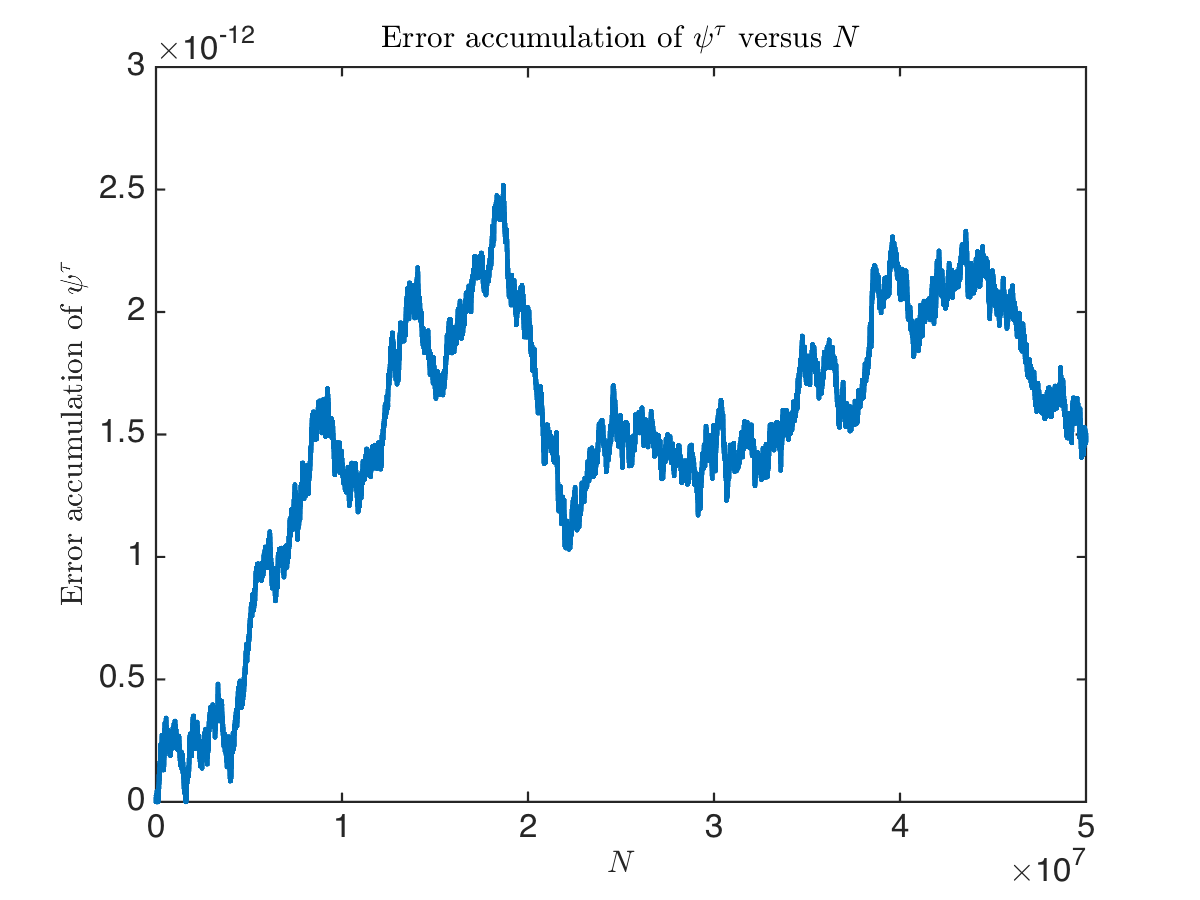}} \quad
	\subfloat[3-step conservative method \eqref{ECDisc3} bootstrapped by RK4 method]{\label{figs:eAccRK_RK}\includegraphics[width=0.48\textwidth]{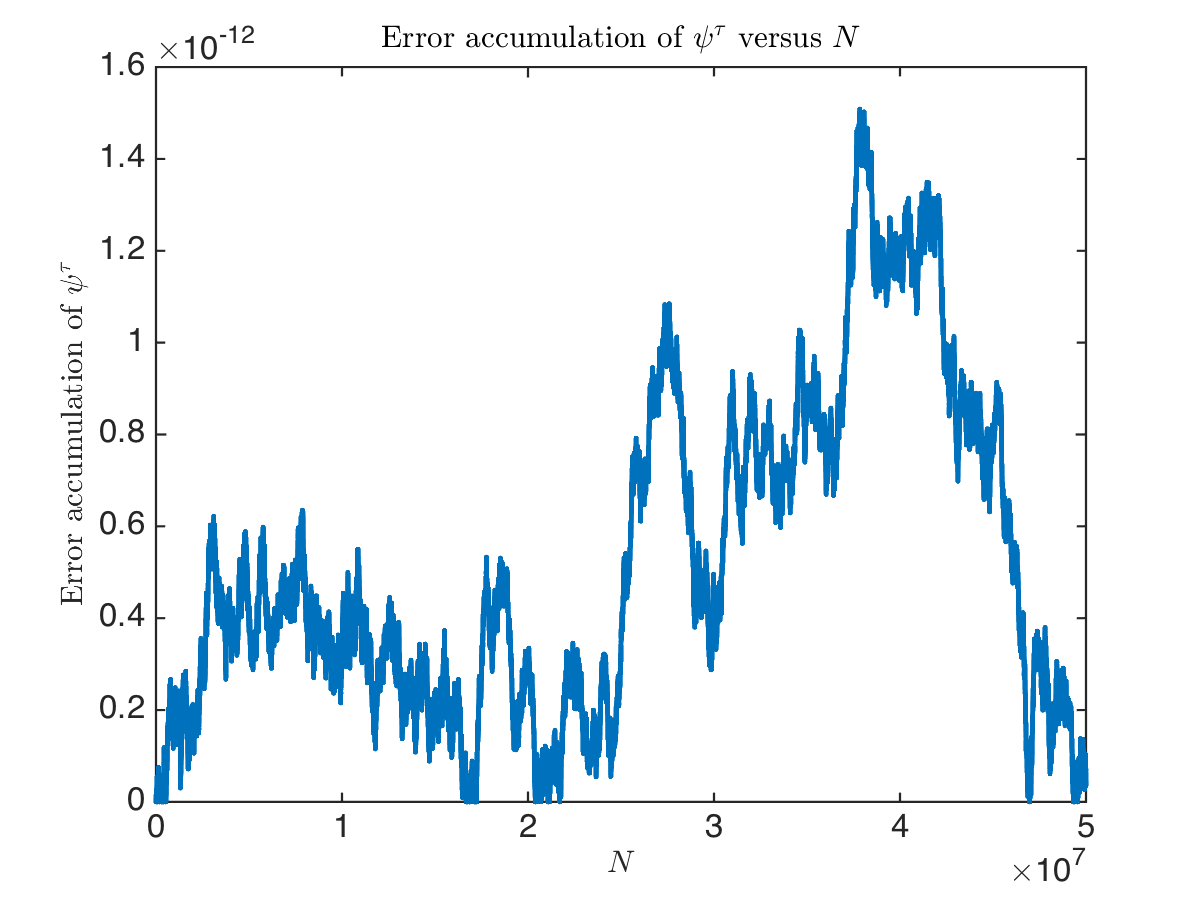}}
	\caption{Error accumulation of the 3-step conservative method with different bootstrapping routines.}
\end{figure}

Similarly, Figures \ref{figs:eAccCM_RK2} and \ref{figs:eAccRK_RK2} illustrates the long term stability holds for the 3-step method on a different set of initial conditions ($x_0 = -0.5$ and $y_0 \approx -0.8717$).

\begin{figure}[tbhp]
\centering
	\subfloat[3-step conservative method \eqref{ECDisc3} bootstrapped by 1-step conservative method \eqref{ECDisc1}]{\label{figs:eAccCM_RK2}\includegraphics[width=0.48\textwidth]{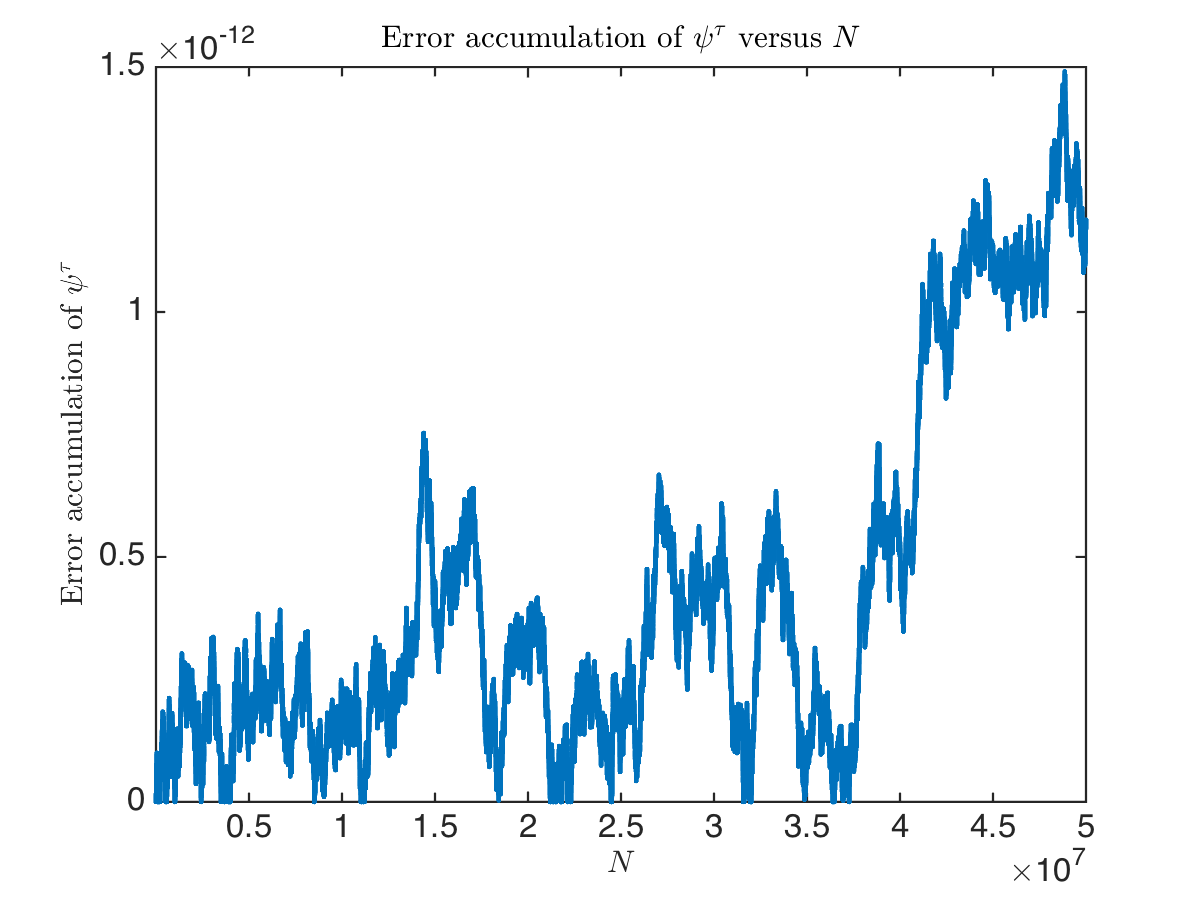}} \quad
	\subfloat[3-step conservative method \eqref{ECDisc3} bootstrapped by RK4 method]{\label{figs:eAccRK_RK2}\includegraphics[width=0.48\textwidth]{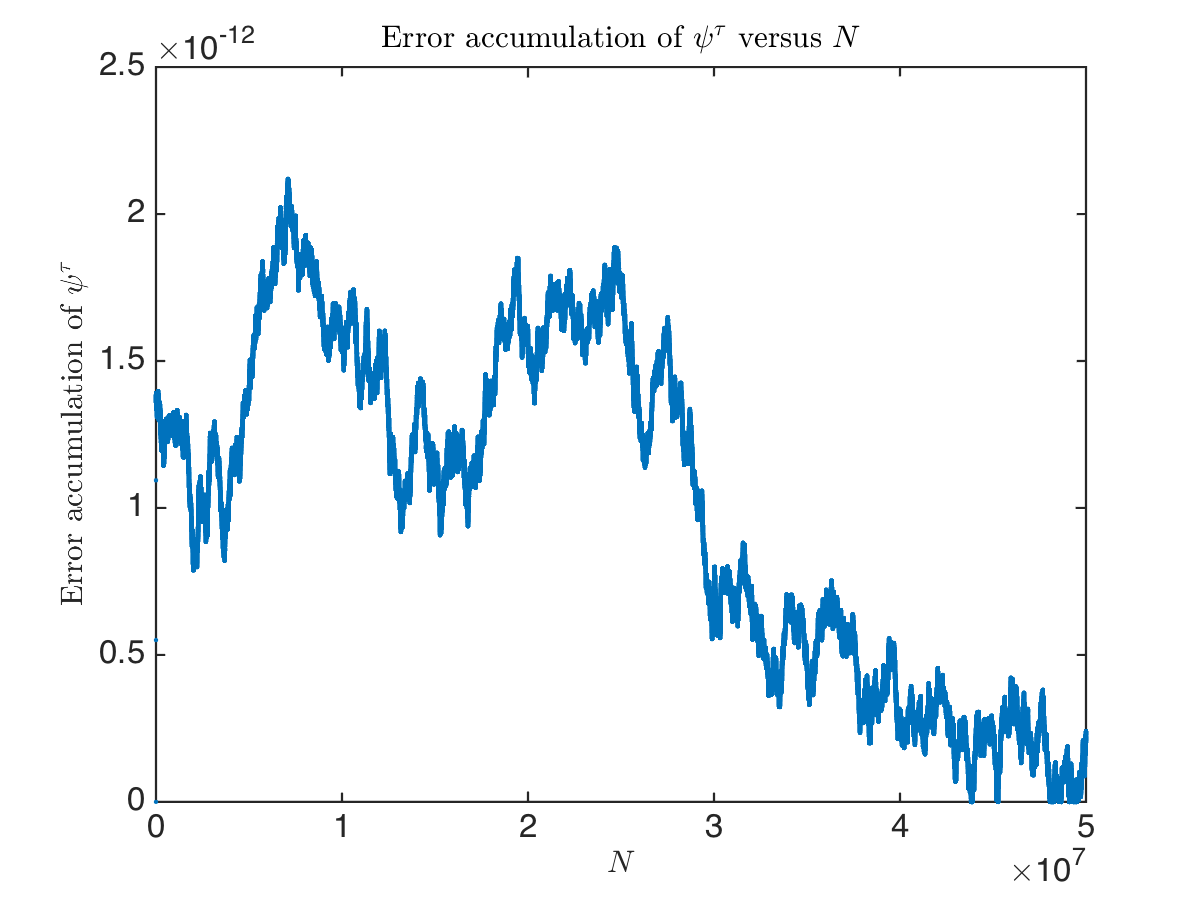}}
	\caption{Error accumulation of the 3-step conservative method with different initial conditions.}
\end{figure}

\section{Conclusion} In this paper, we have presented a long-term stability result for conservative methods with uniformly bounded displacements in the case of autonomous ODEs; specifically the global error is, in principle, bounded for all time. On finite precision machines, the global error is shown to be bounded up to some arbitrarily long time depending only on machine precision and tolerance. Since the main result is mostly based on topological ideas, we believe the stability result can be generalized to certain non-autonomous ODEs and PDEs.

\section*{Acknowledgments}
ATSW would like to thank Siddarth Sankaran for our discussions on algebraic geometry related to this work and Chris Budd for pointing out the connection of the multiplier method with the average vector field method. We thank Alexander Bihlo and Ernst Hairer for the valuable discussions on this work at the Banff International Research Station workshop in June of 2017. We also thank the anonymous reviewers for their helpful comments and suggestions for improving this paper.
\bibliographystyle{siamplain}
\bibliography{refs}

\begin{thebibliography}{10}

\bibitem{ArnFalWin06}
{\sc D.~N. Arnold, R.~S. Falk, and R.~Winther}, {\em {Finite element exterior
  calculus, homological techniques, and applications}}, Acta Numerica,  (2006),
  pp.~1--155.

\bibitem{BenGio94}
{\sc G.~Benettin and A.~Giorgilli}, {\em {On the Hamiltonian interpolation of
  near to the identity symplectic mappings with application to symplectic
  integration algorithms}}, J. Statist. Phys, 74 (1994), pp.~1117--1143.

\bibitem{boch06}
{\sc P.~B. Bochev and J.~M. Hyman}, {\em {Principles of mimetic discretizations
  of differential operators}}, in Compatible spatial discretizations, Springer,
  2006, pp.~89--119.

\bibitem{Bro37}
{\sc D.~Brouwer}, {\em {On the accumulation of errors in numerical
  integration.}}, Astronomical Journal, 46 (1937), pp.~149--153.

\bibitem{BruIavTri15}
{\sc L.~Brugnano, F.~Iavernaro, and D.~Trigiante}, {\em {Analysis of
  Hamiltonian Boundary Value Methods (HBVMs): A class of energy-preserving
  Runge-Kutta methods for the numerical solution of polynomial Hamiltonian
  systems.}}, Commun Nonlinear Sci Numer Simulat, 20 (2015), pp.~650--667.

\bibitem{CalHai95}
{\sc M.~Calvo and E.~Hairer}, {\em {Accurate long-term integration of dynamical
  systems}}, Appl. Numer. Math., 18 (1995), pp.~95--105.

\bibitem{CalSan93}
{\sc M.~Calvo and J.~Sanz-Serna}, {\em {The development of variable-step
  symplectic integrators, with application to the two-body problem}}, SIAM J.
  Sci. Comput., 14 (1993), pp.~936--952.

\bibitem{CelAE12}
{\sc E.~Celledoni, V.~Grimm, R.~McLachlan, D.~McLaren, D.~O'Neale, B.~Owren,
  and G.~R.~W. Quispel.}, {\em {Preserving energy resp. dissipation in
  numerical PDEs using the average vector field method}}, Journal of
  Computational Physics, 231 (2012), pp.~6770--6789.

\bibitem{Cel09}
{\sc E.~Celledoni, R.~I. McLachlan, D.~I. McLaren, B.~Owren, G.~R.~W. Quispel,
  and W.~M. Wright}, {\em {Energy-preserving Runge-Kutta Methods}}, ESAIM:
  M2AN, 43 (2009), pp.~645--649.

\bibitem{ChrMunOwr11}
{\sc S.~H. Christiansen, H.~Munthe-Kaas, and B.~Owren}, {\em {Topics in
  structure-preserving discretization}}, Acta Numerica, 20 (2011), pp.~1--119.

\bibitem{coo87}
{\sc G.~J. Cooper}, {\em {Stability of Runge-Kutta Methods for Trajectory
  Problems}}, IMA J Numer. Anal., 7 (1987), pp.~1--13.

\bibitem{CouFriLew67}
{\sc R.~Courant, K.~Friedrichs, and H.~Lewy}, {\em {On the Partial Difference
  Equations of Mathematical Physics}}, IBM Journal of Research and Development,
  11 (1967), p.~215.

\bibitem{DahOwr11}
{\sc M.~Dahlby and B.~Owren}, {\em {A General Framework for Deriving Integral
  Preserving Numerical Methods for PDEs}}, SIAM J. Sci. Comput., 33 (2011),
  pp.~2318--2340.

\bibitem{DahOwrYag11}
{\sc M.~Dahlby, B.~Owren, and T.~Yaguchi}, {\em {Preserving multiple first
  integrals by discrete gradients}}, J. Phys. A: Math. Theor., 44 (2011).

\bibitem{KanZai95}
{\sc K.~F. and Z.~jiu S.}, {\em {Volume-preserving algorithms for source-free
  dynamical systems }}, Numerische Mathematik, 71 (1995), pp.~451--463.

\bibitem{FurMat10}
{\sc D.~Furihata and T.~Matsuo}, {\em {Discrete Variational Derivative Method:
  A Structure-Preserving Numerical Method for Partial Differential Equations}},
  CRC Press, 2010.

\bibitem{Gon96}
{\sc O.~Gonzalez}, {\em {Time integration and discrete Hamiltonian systems}},
  J. Nonlinear Science, 6 (1996), pp.~449--467.

\bibitem{GraAE04}
{\sc K.~R. Grazier, W.~I. Newman, J.~M. Hyman, P.~W. Sharp, and D.~J.
  Goldstein}, {\em {Achieving Brouwer's law with high-order St\"ormer multistep
  methods}}, ANZIAM J., 46 (2004), pp.~C786--C804.

\bibitem{HaiLubWan06}
{\sc E.~Hairer, C.~Lubich, and G.~Wanner}, {\em {Geometric numerical
  integration: structure-preserving algorithms for ordinary differential
  equations}}, Springer, Berlin, 2006.

\bibitem{HaiMclRaz08}
{\sc E.~Hairer, R.~I. McLachlan, and A.~Razakarivony}, {\em {Achieving
  Brouwer's law with implicit Runge-Kutta methods}}, BIT Numerical Mathematics,
  48 (2008), pp.~231--243.

\bibitem{HaiNorWan00}
{\sc E.~Hairer, S.~P. N{\o}rsett, and G.~Wanner}, {\em Solving Ordinary
  Differential Equations I}, Springer-Verlag Berlin, 2~ed., 1993.

\bibitem{Hir03}
{\sc A.~N. Hirani}, {\em Discrete exterior calculus}, PhD thesis, {C}alifornia
  {I}nstitute of {T}echnology, 2003.

\bibitem{LaBGre75}
{\sc R.~A. LaBudde and D.~Greenspan}, {\em {Energy and momentum conserving
  methods of arbitrary order for the numerical integration of equations of
  motion}}, Numerische Mathematik, 25 (1975), pp.~323--346.

\bibitem{leim04}
{\sc B.~Leimkuhler and S.~Reich}, {\em {Simulating {H}amiltonian dynamics}},
  Cambridge University Press, Cambridge, 2004.

\bibitem{LiVu95}
{\sc S.~Li and L.~Vu-Quoc}, {\em {Finite difference calculus invariant
  structure of a class of algorithms for the nonlinear Klein-Gordon equation}},
  SIAM J. Numer. Anal., 32 (1995), pp.~1839--1875.

\bibitem{MarWes01}
{\sc J.~E. Marsden and M.~West}, {\em {Discrete Mechanics and Variational
  Integrators}}, Acta Numerica,  (2001), pp.~1--158.

\bibitem{MclQui04}
{\sc R.~I. McLachlan and G.~R.~W. Quispel}, {\em {Integral-preserving
  integrators}}, J. Phys. A: Math. Gen., 37 (2004), pp.~L489--L495.

\bibitem{McL99}
{\sc R.~I. McLachlan, G.~R.~W. Quispel, and N.~Robidoux}, {\em {Geometric
  integration using discrete gradients}}, Phil. Trans. R. Soc. Lond., 357
  (1999), pp.~1021--1045.

\bibitem{Mun00}
{\sc J.~R. Munkres}, {\em Topology}, Prentice Hall, 2~ed., 2000.

\bibitem{Olv01}
{\sc P.~J. Olver}, {\em Geometric foundations of numerical algorithms and
  symmetry}, Appl. Alg. Engin. Comp. Commun., 11 (2001), pp.~417--436.

\bibitem{QuaSacSal07}
{\sc A.~Quarteroni, R.~Sacco, and F.~Saleri}, {\em Numerical Mathematics},
  Springer Berlin Heidelberg, 2~ed., 2007.

\bibitem{Qui94}
{\sc G.~D. Quinlan}, {\em {Round-off error in long-term orbital integrations
  using multistep methods}}, Celestial Mech. Dynam. Astronom., 58 (1994),
  pp.~339--351.

\bibitem{QuiMcL08}
{\sc G.~R.~W. Quispel and D.~I. McLaren.}, {\em {A new class of
  energy-preserving numerical integration methods}}, J. Phys. A: Math. Theor.,
  41 (2008).

\bibitem{Qui97}
{\sc G.~R.~W. Quispel and G.~S. Turner}, {\em {Discrete gradient methods for
  solving ODEs numerically while preserving a first integral}}, J. Phys. A:
  Math. Gen., 29 (1996), pp.~L341--L349.

\bibitem{Sha86}
{\sc L.~Shampine}, {\em {Conservation laws and the numerical solution of
  ODEs}}, Comput. Math. Appl., 12B (1986), pp.~1287--1296.

\bibitem{SimTarWon92}
{\sc J.~Simo, N.~Tarnow, and K.~Wong}, {\em Exact energy-momentum conserving
  algorithms and symplectic schemes for nonlinear dynamics}, Computer Methods
  in Applied Mechanics and Engineering, 100 (1992), pp.~63--116.

\bibitem{Wan18}
{\sc A.~T.~S. Wan}, {\em {Higher order conservative methods for dynamical
  systems}}, In preparation.

\bibitem{WanBihNav17a}
{\sc A.~T.~S. Wan, A.~Bihlo, and J.-C. Nave}, {\em {Conservative methods for
  dynamical systems}}, SIAM J. Numer. Anal., 55 (2017), pp.~2255--2285.

\end{thebibliography}

\appendix

\section{Establishing the uniformly bounded displacement property}
In this appendix, we establish the uniformly bounded displacement (UBD) property of Definition \ref{def:LCP} for the 1-step conservative method \eqref{ECDisc1}. As before, we will use $\norm{\cdot}$ to denote the Euclidean norm.

Let $K\subset \mathbb{R}^2$ be a compact subset and $r>0$. To show the 1-step method \eqref{ECDisc1} has uniformly bounded displacements, we need to show that there is a $\tau_c>0$ (depending only on $K$ and $r$) such that if $\tau<\tau_c$ and $\bb x_k \in K$ for each $k\geq 0$, then there is a unique $\bb x_{k+1}$ satisfying $\norm{\bb x_{k+1}-\bb x_k}\leq r$. 

First, we show that there is a unique solution $\bb x_{k+1}$ to \eqref{ECDisc1} in a neighborhood of $\bb x_k$ for sufficiently small $\tau$. By hypothesis of the UBD property, we can assume $\bb x_k\in K$. Then for some $\tau^*$ to be determined, define the map
$T:\overline{B_r(\bb x_k)} \times K \times [0,\tau^*]\rightarrow \mathbb{R}^2$ for any $r>0$ and $a\in \mathbb{R}$ given by,
\begin{equation}
T(\bb x, \bb x_k, \tau) := 
\begin{pmatrix}x_k+\tau(y+y_k) \\ y_k+\tau(x^2+x x_k+x_k^2+a) \end{pmatrix}, \text{ for }\bb x =\begin{pmatrix} x \\ y\end{pmatrix} \text{ and } \tau \in [0,\tau^*]. \label{eq:cMap}
\end{equation}
So to show \eqref{ECDisc1} has a unique solution $\bb x_{k+1}$, it suffices to show that \eqref{eq:cMap} has a unique fixed point $\bb x^*:=\bb x_{k+1}$ by showing $T$ is a contractive map for fixed $\tau$ and $\bb x_k$.
\begin{claim}
There exists a $\tau^*>0$ (depending only on $K$ and $r$) so that $T$ is a contractive map for any fixed $\tau<\tau^*$ and $\bb x_k\in K$.
\end{claim}
\begin{proof}
This follows from standard Banach fixed point type argument. Noting $\bb x_k\in K$ and $\bb x \in \overline{B_r(\bb x_k)} \subset \overline{B_r(K)}$, it follows that
\[
\norm{\bb x_k - T(\bb x, \bb x_k,\tau)} &= \tau \norm{ \begin{pmatrix} y+y_k \\ x^2+xx_k+x_k^2+a\end{pmatrix}} \\
&\leq \tau \underbrace{\max_{\substack{\bb x \in \overline{B_r(K)}\\\bb x_k \in K}} \norm{\begin{pmatrix} y+y_k \\ x^2+xx_k+x_k^2+a\end{pmatrix}}}_{=:M_{K,r}<\infty}.
\] Thus, the image of \eqref{eq:cMap} maps to its domain $\overline{B_r(\bb x_k)}$ if $\tau M_{K,r}\leq r$. Moreover, for $\bb x=\begin{pmatrix} x \\ y\end{pmatrix}, \bb u=\begin{pmatrix} u \\ v\end{pmatrix}$, we have the following estimate,
\begin{align}
\norm{T(\bb x,\bb x_k,
\tau) - T(\bb u,\bb x_k, \tau)} &= \tau\norm{\begin{pmatrix} y-v \\ x^2-u^2+(x-u)x_k\end{pmatrix}} \nonumber \\
&\leq \tau \norm{\begin{pmatrix} 0 & 1 \\ x+u+x_k & 0\end{pmatrix}\begin{pmatrix} x-u \\ y-v\end{pmatrix}} \nonumber \\
&\leq \tau \underbrace{\max_{ \substack{\bb x, \bb u \in \overline{B_r(K)}\\ \bb x_k \in K}}\norm{\begin{pmatrix} 0 & 1 \\ x+u+x_k & 0\end{pmatrix}}}_{=:N_{K,r}<\infty}\norm{\bb x-\bb u} \label{ineq:tDiff}
\end{align} which implies $T$ is contractive provided $\tau N_{K,r}<1$. So picking $\tau^*<\min\{\frac{r}{M_{K,r}}, \frac{1}{N_{K,r}}\}$, the map \eqref{eq:cMap} is contractive for any fixed $\tau \leq \tau^*$ and $\bb x_k \in K$. 
\end{proof}
Thus, we can consider the fixed point as a function $\bb x^*(\bb x_k,\tau)$ for small enough $\tau$. Before showing the UBD property for \eqref{ECDisc1}, we will also need the following claim.
\begin{claim}
$\bb x^*$ is continuous in $\bb x_k \in K$ and locally smooth in $\tau$ provided $\tau<\tau_1$ for some $\tau_1>0$ (depending only on $K$ and $r$). \label{claim:ctsSmo}
\end{claim}
\begin{proof}
First note that by inequality \eqref{ineq:tDiff}, if $\tau < \tau^*$ and $\bb x_k, \bb u_k\in K$,

\[
\norm{\bb x^*(\bb x_k,\tau) - \bb x^*(\bb u_k,\tau)} &= \norm{T(\bb x^*(\bb x_k,\tau),\bb x_k, \tau) - T(\bb x^*(\bb u_k,\tau),\bb u_k, \tau)} \\
&\leq \norm{T(\bb x^*(\bb x_k,\tau),\bb x_k, \tau) - T(\bb x^*(\bb u_k,\tau),\bb x_k, \tau)} \\
&\hskip 5mm+ \norm{T(\bb x^*(\bb u_k,\tau),\bb x_k, \tau) - T(\bb x^*(\bb u_k,\tau),\bb u_k, \tau)} \\
&\leq \tau N_{K,r}\norm{\bb x^*(\bb x_k,\tau)-x^*(\bb u_k,\tau)} \\
&\hskip 5mm+ \norm{T(\bb x^*(\bb u_k,\tau),\bb x_k, \tau) - T(\bb x^*(\bb u_k,\tau),\bb u_k, \tau)} \\
\Rightarrow \norm{\bb x^*(\bb x_k,\tau) - \bb x^*(\bb u_k,\tau)} &\leq \frac{1}{1-\tau N_{K,r}}\norm{T(\bb x^*(\bb u_k,\tau),\bb x_k, \tau) - T(\bb x^*(\bb u_k,\tau),\bb u_k, \tau)}
\]
As $T$ is continuous in $\bb x_k$, the above inequality implies that $\bb x^*$ is continuous in $\bb x_k\in K$. Now to show $\bb x^*$ is locally smooth in $\tau$, consider the function $F: \overline{B_r(\bb x_k)} \times [0,\tau^*]\rightarrow \mathbb{R}^2$ for fixed $\bb x_k \in K$,
\[
F(\bb x,\tau):= \bb x-T(\bb x,\bb x_k,\tau),
\] which is clearly smooth in both $\bb x$ and $\tau$. Thus by implicit function theorem, if the matrix $I-\frac{\partial T}{\partial \bb x}(\bb x,\bb x_k,\tau)$ is invertible at $\bb x=\bb x^*$ and $\tau < \tau^*$, then $\bb x=\bb x^*(\bb x_k,\tau)$ is locally smooth in $\tau$. Indeed,  the matrix is invertible for small enough $\tau$ as follows. Since
\[
I-\frac{\partial T}{\partial \bb x}(\bb x,\bb x_k,\tau) = I-\tau \underbrace{\begin{pmatrix} 0 & 1 \\ 2x+x_k & 0\end{pmatrix}}_{=:A(\bb x, \bb x_k)},
\] and the maximum $\displaystyle \alpha := \max_{\bb x, \bb x_k\in \overline{B_r(K)}} \norm{A(\bb x, \bb x_k)}$ exists, then $\tau \norm{A(\bb x,\bb x_k)} \leq \tau \alpha < 1$ provided if $\tau<\tau_1:=\min\{\tau^*, \frac{1}{\alpha}\}$ and the matrix $I-\frac{\partial T}{\partial \bb x} = I-\tau A$ would be invertible with $\norm{(1-\tau A(\bb x, \bb x_k))^{-1}} \leq \frac{1}{1-\tau \alpha}$.
\end{proof}

It remains to show that \eqref{ECDisc1} has uniformly bounded displacements. By Claim \ref{claim:ctsSmo}, $\bb x^* = \bb x^*(\bb x_k, \tau)$ is continuous in $\bb x_k$ and is $C^1$ (in fact smooth) in $\tau$ for $\tau\leq\tau_1$. So by implicit differentiation in $\tau$,
\[
&\bb x^*(\bb x_k, \tau) = T(\bb x^*(\bb x_k, \tau), \bb x_k, \tau) \\
\Rightarrow& \frac{\partial \bb x^*}{\partial \tau} = \begin{pmatrix} y^*+y_k + \tau \frac{\partial y^*}{\partial \tau} \\ {x^*}^2+x^*x_k+x_k^2+a+\tau\left(2x^*\frac{\partial x^*}{\partial \tau}+\frac{\partial x^*}{\partial \tau} x_k\right)\end{pmatrix} \\
\Rightarrow& (I-\tau A(\bb x^*,\bb x_k))\frac{\partial \bb x^*}{\partial \tau} = \begin{pmatrix} y^*+y_k \\ {x^*}^2+x^* x_k + x_k^2+a\end{pmatrix}.
\] Moreover, since $\bb x^*(\bb x_k,\tau) \in \overline{B_r(\bb x_k)} \subset \overline{B_r(K)}$, then, as before in the proof of Claim \ref{claim:ctsSmo}, $I-\tau A(\bb x^*,\bb x_k)$ is invertible with $\norm{(1-\tau A(\bb x^*,\bb x_k))^{-1}} \leq \frac{1}{1-\tau \alpha}$ if $\tau\leq\tau_1$. Thus, the derivative of $\bb x^*$ with respect to $\tau$ can be bounded as,
\[
\norm{\frac{\partial \bb x^*}{\partial \tau}} &\leq \norm{(1-\tau A(\bb x^*,\bb x_k))^{-1}}\norm{\begin{pmatrix} y^*+y_k \\ {x^*}^2+x^* x_k + x_k^2+a\end{pmatrix}} \\
&\leq \frac{1}{1-\tau \alpha} \underbrace{\max_{\tau \in [0,\tau_1]}\max_{\bb x_k\in K}\norm{\begin{pmatrix} y^*+y_k \\ {x^*}^2+x^* x_k + x_k^2+a\end{pmatrix}}}_{=: L_K<\infty} \leq \frac{L_K}{1-\tau_1 \alpha}.
\]
Finally to show the UBD property, since $\bb x_{k+1}=\bb x^*(\bb x_k, \tau)$ and $\bb x_k = \bb x^*(\bb x_k,0)$, the displacement between $\bb x_{k+1}$ and $\bb x_{k}$ can be uniformly bounded by any $r>0$ as
\[
\norm{\bb x_{k+1}-\bb x_k} &= \tau\norm{ \frac{\bb x^*(\bb x_k, \tau) - \bb x^*(\bb x_k, 0)}{\tau}} \leq \tau \max_{\tau \in [0,\tau_1]}\max_{\bb x_k\in K} \norm{\frac{\partial \bb x^*}{\partial \tau}(\bb x_k, \tau)}<r,
\] provided if $\tau<\tau_c:=\min\{\tau_1, \frac{r(1-\tau_1 \alpha)}{L_K}\}$. This shows \eqref{ECDisc1} has the uniformly bounded displacement property for any compact subset $K\subset \mathbb{R}^2$ and $r>0$.

\section{Numerical verification of convergence order of three conservative methods}
\label{sec:convOrder}

Tables \ref{tab:conv1}-\ref{tab:conv3} shows the convergence order of the three conservative methods \eqref{ECDisc1}, \eqref{ECDisc2}, \eqref{ECDisc3} for the variable $y$ at a fixed time $T=N\tau$, where the initial conditions were chosen to be $(x_0,y_0)\approx(0.571,-8.331\times10^{-3})$ and other initial values were bootstrapped using the standard 4-th order Runge-Kutta method. We observed that both the 1-step and 2-step method were second-order accurate and the 3-step method was third order accurate.

\begin{table}[H]
\label{tab:conv1}
\noindent \begin{centering}
\begin{tabular}{|c|c|c|c|c|}
\hline 
$\tau$ & $N$ & $y_N^\tau$ & $y_N^\tau-y_N^{\tau/2}$ & $\log_2\left[\dfrac{y_N^\tau-y_N^{\tau/2}}{y_N^{\tau/2}-y_N^{\tau/4}}\right]$
\tabularnewline
\hline 
\hline 
$1.000\times10^{-2}$ & 200 & -0.815259441420454 & $4.0669\times10^{-5}$ & 1.999787\tabularnewline
\hline 
$5.000\times10^{-3}$& 400 & -0.815218772324505 & $1.0169\times10^{-5}$ & 1.999947\tabularnewline
\hline 
$2.500\times10^{-3}$ & 800 & -0.815208603548917 & $2.5423\times10^{-6}$ & 1.999987\tabularnewline
\hline 
$1.250\times10^{-3}$ & 1600 & -0.815206061261177 & $6.3558\times10^{-7}$ & 1.999997\tabularnewline
\hline 
$6.250\times10^{-4}$ & 3200 & -0.815205425683374 & $1.5889\times10^{-7}$ & -\tabularnewline
\hline 
$3.125\times10^{-4}$ & 6400 & -0.815205266788567 & - & -\tabularnewline
\hline 
\end{tabular}
\par\end{centering}
\caption{Second order convergence of the 1-step conservative method of \eqref{ECDisc1}.}\vskip -5mm
\end{table}

\begin{table}[H]
\label{tab:conv2}
\noindent \begin{centering}
\begin{tabular}{|c|c|c|c|c|}
\hline 
$\tau$ & $N$ & $y_N^\tau$ & $y_N^\tau-y_N^{\tau/2}$ & $\log_2\left[\dfrac{y_N^\tau-y_N^{\tau/2}}{y_N^{\tau/2}-y_N^{\tau/4}}\right]$
\tabularnewline
\hline 
\hline 
$1.000\times10^{-2}$ & 200 & -0.815422021628171 & $1.6258\times10^{-4}$ & 1.99916\tabularnewline
\hline 
$5.000\times10^{-3}$& 400 & -0.815259441420454 & $4.0669\times10^{-5}$ & 1.99979\tabularnewline
\hline 
$2.500\times10^{-3}$ & 800 & -0.815218772324505 & $1.0169\times10^{-5}$ & 1.99995\tabularnewline
\hline 
$1.250\times10^{-3}$ & 1600 & -0.815208603548917 & $2.5423\times10^{-6}$ & 1.99999\tabularnewline
\hline 
$6.250\times10^{-4}$ & 3200 & -0.815206061261177 & $6.3558\times10^{-7}$ & -\tabularnewline
\hline 
$3.125\times10^{-4}$ & 6400 & -0.815205425683374 & - & - \tabularnewline
\hline 
\end{tabular}
\par\end{centering}
\caption{Second order convergence of the 2-step conservative method of \eqref{ECDisc2}.}\vskip -5mm
\end{table}

\begin{table}[H]
\label{tab:conv3}
\noindent \begin{centering}
\begin{tabular}{|c|c|c|c|c|}
\hline 
$\tau$ & $N$ & $y_N^\tau$ & $y_N^\tau-y_N^{\tau/2}$ & $\log_2\left[\dfrac{y_N^\tau-y_N^{\tau/2}}{y_N^{\tau/2}-y_N^{\tau/4}}\right]$
\tabularnewline
\hline 
\hline 
$1.000\times10^{-2}$ & 200 & -0.815187526809099 & $1.4491\times10^{-5}$ & 2.9320\tabularnewline
\hline 
$5.000\times10^{-3}$& 400 & -0.815203036844703 & $1.8988\times10^{-6}$ & 2.9659\tabularnewline
\hline 
$2.500\times10^{-3}$ & 800 & -0.815204935644390 & $2.4302\times10^{-7}$ & 2.9829\tabularnewline
\hline 
$1.250\times10^{-3}$ & 1600 & -0.815205178665148 & $3.0739\times10^{-8}$ & 2.9911\tabularnewline
\hline 
$6.250\times10^{-4}$ & 3200 & -0.815205209404162 & $3.8661\times10^{-9}$ & -\tabularnewline
\hline 
$3.125\times10^{-4}$ & 6400 & -0.815205213270296 & - & -\tabularnewline
\hline 
\end{tabular}
\par\end{centering}
\caption{Third order convergence of the 3-step conservative method of \eqref{ECDisc3}.}\vskip -5mm
\end{table}

\end{document}